\documentclass[english]{article}
\usepackage{a4wide}
\usepackage[utf8]{inputenc}
\usepackage[T1]{fontenc}
\usepackage{babel}
\usepackage{color}
\usepackage{amsmath, amsfonts, amsthm}
\usepackage{graphicx}
\usepackage{bm}

\newtheorem{Theorem}{Theorem}[section]
\newtheorem{Proposition}{Proposition}[section]
\newtheorem{Lemma}{Lemma}[section]
\newtheorem{Remark}{Remark}[section]
\newtheorem{Corollary}{Corollary}[section]

\numberwithin{equation}{section}

\newcommand{\bTheorem}[1]{
\begin{Theorem} \label{T#1} }
\newcommand{\eT}{\end{Theorem}}

\newcommand{\bProposition}[1]{
\begin{Proposition} \label{P#1}}
\newcommand{\eP}{\end{Proposition}}

\newcommand{\bLemma}[1]{
\begin{Lemma} \label{L#1} }
\newcommand{\eL}{\end{Lemma}}

\newcommand{\bCorollary}[1]{
\begin{Corollary} \label{C#1} }
\newcommand{\eC}{\end{Corollary}}
\newcommand{\bFormula}[1]{
\begin{equation} \label{#1}}
\newcommand{\eF}{\end{equation}}

\newcommand{\Div}       {{\rm div}_x}

\newcommand{\dx}        {\,{\rm d}\x}
\newcommand{\dt}        {\,{\rm d}t}
\newcommand{\dxdt}  {\,\dx \, \dt}

\newcommand{\B}     {{\mathcal B}}
\newcommand{\F}     {{\mathcal F}}

\newcommand{\ep}        {\varepsilon}

\newcommand{\bY} {{\bf Y}}
\newcommand{\bZ} {{\bf Z}}
\newcommand{\bU} {{\bf U}}

\newcommand{\Grad}  {\nabla_x}
\newcommand{\vr}        {\varrho}

\newcommand{\vu}        {\vc{u}}
\newcommand{\vphi}  {\pmb{\varphi}}

\newcommand{\Om}{\Omega}

\newcommand{\vn}{\vc{n}}
\newcommand{\vU}{\vc{U}}
\newcommand{\vw}{\vc{w}}

\newcommand{\vF}{\vc{F}}
\newcommand{\vV}{\vc{V}}
\newcommand{\dv}{{\rm div}}

\renewcommand{\d}{{\rm d} }

\newcommand{\R}{{\mathbb R}}

\newcommand{\de}{\partial}

\newcommand{\vc}[1]     { {\bf #1} }
\newcommand{\tn}[1]     {\mathbb{#1} }
\newcommand{\x}{{\mathbf x}}
\newcommand{\X}{{\mathbf X}}
\newcommand{\y}{{\mathbf y}}

\newcommand{\ZZ}{\widetilde{\bZ}_2}
\newcommand{\ZZj}{\widetilde{\bZ}_1}

\newcommand{\Ot}{\widetilde{\tn{O}}}

\begin{document}

\title{\bf Weak-strong uniqueness for the compressible fluid-rigid body interaction}

\author{ 
Ond\v rej Kreml\textsuperscript{1}, \v S\'arka Ne\v casov\'a\textsuperscript{1{*}}, Tomasz Piasecki\textsuperscript{2} 
}
\date{}

\maketitle

\begin{center}
\textsuperscript{1}Institute of Mathematics, Czech Academy of Sciences \\
 \v Zitn\'a 25, 115 67 Praha 1, Czech Republic.\\
 \textsuperscript{2} Institute of Applied Mathematics and Mechanics, University of Warsaw\\
Banacha 2, 02-097 Warszawa, Poland.\\
 \textsuperscript{*}Corresponding author

\end{center}

\vspace{1cm}

\begin{abstract}
In this work we study the coupled system of partial and ordinary differential equations describing the interaction between a compressible isentropic viscous fluid and a rigid body moving freely inside the fluid. In particular the position and velocity of the rigid body in the fluid are unknown and the motion of the rigid body is driven by the normal stress forces of the fluid acting on the boundary of the body. We prove that the strong solution, which is known to exist under certain smallness assumptions, is unique in the class of weak solutions to the problem. The proof relies on a correct definition of the relative energy, to use this tool we then have to introduce a change of coordinates to transform the strong solution to the domain of the weak solution in order to use it as a test function in the relative energy inequality. Estimating all arising terms we prove that the weak solution has to coincide with the transformed strong solution and finally that the transformation has to be in fact an identity. 

\end{abstract}


\section{ Introduction}

 In this article we investigate the weak-strong uniqueness of the problem of motion of a rigid body inside a viscous
compressible fluid. The
fluid and the body occupy a bounded domain $\Omega\subset \R^{3}$. 
 
Let us denote by $\mathcal{B}(t)\subset
\Omega$ a bounded domain occupied by the rigid body and a domain filled
by the fluid by $\mathcal{F}(t)=\Omega\setminus \overline{\mathcal{B}(t)
}$ at a time moment $t\in \R^{+}.$\ \ Assuming that the initial
position $\mathcal{B}(0)$ of the rigid body is prescribed, for simplicity of
notation we denote $\mathcal{B}_{0}=\mathcal{B}(0)$ and $\mathcal{F}_{0}=
\mathcal{F}(0)$. The interface between the body and the fluid is denoted
by $\partial \mathcal{B}(t)$, the normal vector to the boundary is denoted
by $\vn(t)$ and it is pointing outside $\Omega$ and inside $
\mathcal{B}(t)$. For $T > 0$ we introduce the following notation for the space-time cylinders
\begin{equation*}
\begin{array}{l}
{Q}_{\mathcal{F}}= \bigcup_{t\in (0,T)} \{t\} \times \mathcal{F}(t),

\\
{Q}_{\partial \mathcal{B}}=\bigcup_{t\in (0,T)} \{t\} \times \partial\mathcal{B}(t),

\\
{Q}_{\mathcal{B}}=\bigcup_{t\in (0,T)} \{t\} \times \mathcal{B}(t),

\\
Q_T= (0,T) \times \Omega.
\end{array}
\end{equation*}%
The fluid motion is governed by the equations
\begin{equation}
\left\{
\begin{array}{ccc}
\partial_t \vr_{\mathcal{F}} + \mathrm{div}_x(\,\vr_{\mathcal{F}} \vu_{\mathcal{F}})=0 & \mathrm{in} & Q_{%
\mathcal{F}}, \\
\partial_{t}(\vr_{\mathcal{F}} \vu_{\mathcal{F}}) + \mathrm{div}_x(\vr_{\mathcal{F}}\vu_{\mathcal{F%
}}\otimes \vu_{\mathcal{F}}) + \nabla_x p(\vr_{\mathcal{F}}) - \mathrm{div}_x\,\tn{S}(\nabla_x
\vu_{\mathcal{F}})=0 & \mathrm{in} & {Q}_{%
\mathcal{F}},\\
\vu_{\mathcal{F}} = 0 & \mathrm{on} & (0,T) \times \partial\Omega , \\
\vu_{\mathcal{F}}=\vu_{\mathcal{B}} 
&\mathrm{on} & Q_{\partial
\mathcal{B}}, \\
(\vr_\F \vu_{\mathcal{F}})(0)= (\vr_\F\vu_\F)_{0} & \mathrm{in} & \mathcal{F%
}_{0},\\
\vr_{\mathcal{F}}(0)=\vr_{\mathcal{F}0} & \mathrm{in} & \mathcal{F}_{0},
\end{array}%
\right.  \label{fluidmotion}
\end{equation}%
where $\vu_{\mathcal{F}}$ and $\vr_{\mathcal{F}}$ denote the
velocity and the density of the fluid and $\vu_{\mathcal{B}}$ is
the full velocity of the rigid body. For simplicity of presentation we assume that the external body forces acting on the fluid are zero. We recall that the rate of the strain
tensor of the fluid and its stress tensor are defined by
\begin{equation*}
\tn{D}(\vu_{\mathcal{F}})=\frac{1}{2}(\nabla_x \vu_{%
\mathcal{F}} + (\nabla_x \vu_{\mathcal{F}})^{T})\qquad \text{and}%
\qquad \tn{S}(\nabla_x \vu_{\mathcal{F}})=2\mu \,%
\tn{D}(\vu_{\mathcal{F}}) + \lambda \mathrm{div}_x \vu_\F\tn{I},
\end{equation*}%
with $\mu,\lambda$ being the viscosity coefficients of the fluid such that $\mu > 0$, $\mu+\lambda \geq 0$. We assume that the pressure is given by
\begin{equation}
p(\vr_\F) = a\vr_{\mathcal{F}}^\gamma \label{pres}
\end{equation}
for some $\gamma > 1$ and again for simplicity we set $a = 1$ for the rest of the paper. The pressure potential $P(\vr)$ is defined as
\begin{equation}
P'(\vr_\F)\vr_\F - P(\vr_\F) = p(\vr_F) \label{prespot},
\end{equation}
whence $P(\vr_\F) = \frac{\vr_\F^\gamma}{\gamma-1}$ for the pressure law \eqref{pres}.

We denote the density of the body by $\vr_\B$ and define the mass of the body $m$, its center of mass $\X$ and its inertial tensor $\tn{J}$ by, respectively,
\begin{align}
    m &= \int_{\B(t)} \vr_\B(t,\x) \dx, \label{def:m} \\
    \X(t) &= \frac{1}{m} \int_{\B(t)} \vr_\B(t,\x)\x \dx, \label{def:X} \\
    \tn{J}(t)\vc{a}\cdot\vc{b} &= \int_{\B(t)} \vr_\B(t,\x) \left[\vc{a}\times (\x-\X(t))\right]\cdot\left[\vc{b}\times (\x-\X(t))\right] \dx. \label{def:J}
\end{align}
Note that without loss of generality we can assume that the center of mass of the body is at the origin at time zero, i.e. $\X(0) = 0$. The velocity of the rigid body is given by
\begin{equation}\label{eq:rigidbodyvelocity}
\vu_{\mathcal{B}}(t,\x)=\vV(t) + \vw(t)\times (\x-\X(t))    
\end{equation}
for any $(t,\x) \in Q_\B$.

The fluid equations are coupled to the following balance equations for the
translation velocity $\vV$ and the angular velocity $\vw$ of the body,
\begin{equation}
\left\{
\begin{array}{l}\displaystyle
m\frac{\d}{\dt}\vV(t) = - \int_{\partial\mathcal{B}(t)} (\tn{S}(\nabla_x\vu_\F) - p(\vr_\F)\tn{I})\vc{n} \d S\quad
 \text{ in } (0,T), \\ \displaystyle
\tn{J}(t)\frac{\d}{\dt}\vw(t) = \tn{J}(t)\vw(t)\times\vw(t) - \int_{\partial\mathcal{B}(t)} (\x-\X(t))\times (\tn{S}(\nabla_x\vu_\F) - p(\vr_\F)\tn{I})\vc{n} \d S\quad \text{ in } (0,T),\\
\vV(0)=\vV_{0},\qquad \vw(0)= \vw_{0}. 
\end{array}
\right.  \label{bodymotion}
\end{equation}

\subsection{The weak formulation}
Let us define the weak solution similarly as in the work of Feireisl \cite{F4}. We denote

\begin{equation}
\vr  =
\left| 
\begin{array}{c}
\vr_{\mathcal{F}} \ \ \mbox{in}\ Q_{\F},\\\
\vr_{\mathcal{B}}\ \ \mbox{in}\ Q_{\B},
\end{array}
\right.
\ \ \vu =
\left|
\begin{array}{c}
\vu_{\mathcal{F}}\ \ \mbox{in}\ Q_{\F},\\\
\vu_{\mathcal{B}}\ \ \mbox{in}\ Q_{\B},
\end{array}
\right.
\label{global}
\end{equation}
and then we can write the weak formulation for $(\vr,\vu)$ on the whole space-time domain $Q_T$.
Indeed, since $\vu_\B$ is the velocity describing the movement of the domain $\B(t)$, $\vr_B$ is as smooth as the initial density $\vr_B(0)$ and $\vu_\B$ is smooth in space because it is a rigid motion. Therefore the conservation of mass of the body implies
\begin{equation} \label{eq:CE_body}
\int_0^\tau \int_{\B(t)} (\de_t \vr_\B + \Div(\vr_B\vu_\B))\phi \dx \dt = 0
\end{equation}
for any test function $\phi \in C^\infty([0,T]\times \R^3)$.
As a result, the weak form of the continuity equation (\ref{fluidmotion})$_2$ on the fluid part together with \eqref{eq:CE_body} can be written in the concise form
\begin{equation}
{\displaystyle \int_0^T \int_{\Omega}\vr \partial_t \phi + \vr
\vu\cdot\nabla_x \phi \dx \dt = - \int_\Omega \vr_0\phi(0,\cdot) \dx} \label{continu}
\end{equation}
for any test function $\phi\in C^\infty_c([0,T)\times \Omega)$. Here $\vr_0 = \vr_{\F 0}$ on the fluid part $\F_0$ and $\vr_0 = \vr_\B$ on $\B_0$.

Similarly the momentum equation (\ref{fluidmotion})$_1$ on the fluid part can be combined with the balance equations for the translation and angular velocities of the body to a single equation
\begin{equation}
{\displaystyle \int_0^T \int_{\Omega}(\vr \vu)\cdot \partial_t
\vphi + [\vr \vu\otimes \vu]: \tn{D}(\vphi)
+p\mbox{div}_x\vphi - \tn{S}(\nabla_x\vu): \tn{D}(\vphi) \dx \dt = -\int_\Omega (\vr\vu)_0\cdot\vphi(0,\cdot) \dx}
\label{moment}
\end{equation}
for any test function  $\vphi \in {\mathcal R}(Q_{\B})$, where
$${\mathcal R}(Q_{\B})\equiv \{\vphi\in C^\infty_c([0,T)\times \Omega)\ |\ \tn{D}(\vphi)=0\ \ \mbox{on an open neighbourhood of}\ \overline{Q_{\B}}\}.$$
This involves several operations which will be useful also later, so we demonstrate how we obtain the equation \eqref{moment} in the Appendix.

The motion of the rigid body is described through a family of isometries of $\R^3$ by
$$
\eta [t,s]:  \R^3 \to \R^3,\ \ {\overline \B}(t)= \eta[t,s]({\overline \B}(s))\ \ \mbox{for}\ 0\leq s\leq t\leq T,
$$
or equivalently
$$
\eta[t,s]=\eta[t,0]\left(\eta[s,0]\right)^{-1},
$$
where the mapping $\eta[t,0]$ satisfies
\[
\eta[t,0](\x)\equiv \eta[t](\x) = \X(t) + \tn{O}(t) \x,\ \ \tn{O}(t)\in SO(3).
\]

We say that the velocity $\vu$ is compatible with the family
$\{\B,\eta\},$  if the
function $t\mapsto \eta[t](\x)$ is absolutely continuous
on $[0,T]$ for any $\x\in \R^3$ and if
\begin{equation}
\left( \frac{\partial}{\partial t}\ \eta[t] \right)
\left(\left(\eta[t]\right)^{-1}(\x)\right)=\vu(t,\x)\
 \ \ \mbox{for}\ \x\in {\overline {\B}}(t)\ \mbox{and a.e.}\ t\in (0,T).
\label{etai}
 \end{equation}
In other words if
$$
\vu(t,\x)=\vu_\B(t,\x)\equiv \vV(t)+{\tn{Q}}(t)(\x-\X(t))\ \ \mbox{for}\ \x\in {\overline {\B}}(t)\ \mbox{and a.e.}\ t\in (0,T),
$$
 where
\begin{equation}\label{eq:QO}
    \vV(t)= \frac{\d}{\d t} \X(t),\ \ \tn{Q}(t) = \left(\frac{\d}{\d t}\tn{O}(t)\right)\left(\tn{O}(t)\right)^{-1}
\end{equation}
for a.e. $t \in (0,T)$. Note that the matrix $\tn{Q}(t)$ is skew-symmetric, so the term $\tn{Q}(\x-\X)$ can be written as $\vw \times (\x-\X)$ for a uniquely determined vector $\vw$.

 We define the weak solution of the studied problem \eqref{fluidmotion}-\eqref{bodymotion} in the spirit of \cite{F4}. We say that functions $(\vr,\vu)$ and a family $\{\overline{\B},\eta\}$ represent a weak solution on the set $(0,T)\times\Omega$ if
\begin{itemize}
    \item $\vr \geq 0$, $\vr, P(\vr) \in L^\infty(0,T,L^1(\Omega))$, $\vu \in L^2(0,T,W^{1,2}_0(\Omega))$,
    \item the functions $(\vr,\vu)$ satisfy the continuity equation \eqref{continu} together with its renormalized form
    \begin{equation}
{\displaystyle \partial_t b(\vr)+\Div(b(\vr) \vu)+\left(\vr b'(\vr)- b(\vr)\right)  \Div \vu=0\ \
\mbox{in}\ {\mathcal D}'([0,T)\times \Omega) \quad \forall b \in C^1([0,+\infty))} \label{renor}
\end{equation}
    \item the momentum equation \eqref{moment} is satisfied for any test function $\vphi \in \mathcal{R}(Q_\B)$,
    \item the energy inequality
    \begin{equation}
 \int_{\Omega} \left[ \frac{1}{2}\vr|\vu|^2+P(\vr)\right](\tau,\cdot) \dx
+\int_0^\tau\int_{\Omega} \tn{S}(\nabla_x\vu):\nabla_x\vu \dx \dt \leq E_0,
\label{energy}
\end{equation}
holds for a.e. $\tau \in [0,T]$, where $E_0 = \int_\Omega \frac{1}{2}\frac{|(\vr\vu)_0|^2}{\vr_0} + P(\vr_0) \dx$,
\item the velocity $\vu$ is compatible with $\{\overline{\B},\eta\}$ and the functions $\eta[t]: \R^3 \mapsto \R^3$ are isometries.
\end{itemize}


We also state the following remarks stated already in \cite{F4}.
\smallskip

{\bf Remarks:}
\begin{itemize}
\item In the class of weak solutions the velocity $\vu$  is
not necessarily continuous function. The meaning $\vu(t,\x)=\vu_\B(t,\x)$ for $\x\in \overline{\B}(t)$ is to be understood in the sense that the difference
    $$ \vu(t,.)  - \vu_\B(t,.) \mbox { belongs locally
 to the space } W^{1,2}_0 (  \R^3 \setminus \overline{{\B}}(t)).$$
 
 \item There is no a priori reason that the momentum $(\vr\vu)$ is continuous in time we can only have that function
     \begin{equation}\label{eq:property}
     t \mapsto \int_{\R^3} (\vr \vu)(t,\x)\cdot \vphi(\x) \dx \qquad \text{ is continuous}    
     \end{equation}
 in a certain neighbourhood of a point $t_0$ provided
$\vphi = \vphi(\x) \in {\mathcal D}(\Omega)$ and $\tn{D}(\vphi)=0$
on a neighbourhood of $\overline{\B}(t_0)$. 

\item The alternative condition to the concept of
compatibility of the velocity $\vu$ with the rigid objects was used
in \cite{DEES2}, \cite{SST}, where the authors assumed
$$
\vu\in L^2((0,T);W^{1,2}_0\cap V^s(\Omega)),
$$ 
where the sets
$V^s=V^s(t)$ are defined as
$$
V^s=\{\vu \in W^{1,2}(\Omega)\ | \ \tn{D}(\vu)\vr_{\B}(t)=0\}.
$$

 
\end{itemize}

\subsection{Discussion and main result}
The existence theory of both weak and strong solutions for systems describing the motion of rigid body in a viscous incompressible fluid is nowadays well developed. First results on the existence of weak solutions until first collision go back to the work of  Conca, Starovoitov and Tucsnak \cite{CST}, Desjardins and Esteban \cite{DEES1}, Gunzburger, Lee and Seregin \cite{GLSE}, Hoffman and Starovoitov \cite{HOST}. After that the possibility of collision  in case of weak solution was included, see \cite{SST},  \cite{F3}.  The problem of collisions was investigated in several papers, see \cite{HES, HIL, HT}, where it was shown that under non-slip boundary conditions and under sufficient regularity of the boundary of a domain and rigid body the collisions cannot occur. After that  G\'erard-Varet and Hillairet have shown that to model collisions we need to assume less regularity of domain or different boundary conditions, see e.g. \cite{GH,GH2,GHC}. In case of very high viscosity, under the assumption that rigid bodies are not touching each other or not touching the boundary at initial time it has been shown that the collisions cannot occur in finite time, see \cite{FHN}.
For introduction to the problem of fluid coupled with rigid body see \cite{G2}, \cite{SER3}.
Let us also mention results on strong solutions, see e.g. \cite{T}, \cite{Wa}, \cite{GGH13}. 
 
Much less results are available on the motion of rigid structure in a compressible fluid. The existence of strong solutions in $L^2$ framework for small data up to a collision has been shown in \cite{BG} (see Theorem \ref{t2} below). A different approach to the existence of strong solutions in $L^p$ setting based on $\mathcal{R}$-bounded operators has been applied in \cite{HiMu}, \cite{HMTT} in the barotropic case and \cite{MaityTucsnak18} for the full system.

The existence of weak solution, also up to a collision but without smallness assumptions has been shown in \cite{DEES2}. Generalization of this result allowing collisions has been given in \cite{F4}, see Theorem \ref{t1} below.   

Let us now quote more precisely two existence results, already mentioned before, which are most relevant for our study. The first one concerns  
global existence of weak solutions to the studied problem and has been shown in \cite{F4}. 

\begin{Theorem} \label{t1}
 Suppose that $\Omega\subset \R^3$ and let initial data
$\vr_0$, $(\vr\vu)_0$ be such that
\begin{equation}
\vr_0\geq 0,\ \ \ \vr_0\in L^{\gamma}(\Omega),
\label{roz}
\end{equation}
\begin{equation}
(\vr\vu)_0=0\ \ \mbox{a.e.  on the set}\ \{\x\in\Omega\ |\ \vr_0=0\},\ \ \ \frac{|(\vr\vu)_0|^2}{\vr_0}\in L^1(\Omega).
\label{vz}
\end{equation}
Let $\overline{\B}_0$ be a compact connected set with nonempty interior such that $|\de\overline{\B}_0| = 0$. Suppose moreover that $\gamma>3/2$ and that $\mu > 0$, $\mu+\lambda \geq 0$. Let $T > 0$.

Then the problem
\eqref{fluidmotion}-\eqref{bodymotion} admits a weak solution $(\vr,\vu)$ and $\{
\overline{\B},\eta\}$ on $Q_T$.
\end{Theorem}

\begin{Remark}
Note that the original result of \cite{F4} uses the weak formulation of the equations without the initial data terms and with test functions with compact support on the time interval $(0,T)$. The inclusion of initial data terms in \eqref{continu}, \eqref{moment} can be justified using \eqref{eq:property}.
\end{Remark}

Next we recall the result proved in \cite{BG}, stating the existence of 
strong solutions to the studied problem. We define $\vu_{\F 0} := \frac{(\vr_\F\vu_F)_0}{\vr_{\F 0}}$. The strong solution is formulated in terms of $(\vr_\F,\vu_\F)$ for the fluid part and $(\X,\vw)$ such that $\vV = \frac{\d \X}{\dt}$ for the solid part and the position of the rigid body. The initial data for the problem are therefore $(\vr_{\F 0},\vu_{\F 0})$, $\X(0) = 0$, $\frac{\d}{\dt}\X(0) = \vV_0$, $\vw(0) = \vw_0$. 

\begin{Theorem} \label{t2}
Let $\bar \vr$ be the mean value of $\vr_{\F 0}$ in $\F_0$.
Assume that $d(\B_0,\de \Omega) > 0$
and the initial conditions satisfy appropriate compatibility conditions, see \cite[formulas (15), (16)]{BG}.  
Then there exists a constant $\delta>0$ such that if 
\begin{equation}\label{eq:strongass}
\|\vr_{\F 0}-\bar \vr\|_{H^3(\F_0)}+\|\vu_{\F 0}\|_{H^3(\F_0)}+|\vV_0|+|\vw_0|<\delta,    
\end{equation}
then the system \eqref{fluidmotion}-\eqref{bodymotion} admits a unique strong solution 
$(\vr_\F,\vu_\F,\X,\vw)$ defined on $(0,T)$ for all $T$ such that 
\begin{equation}\label{eq:dist_cond}
d(\B(t),\de \Omega)\geq \kappa>0 \quad \forall t\in[0,T] \quad \textrm{for some}\; d(\B_0,\de \Omega)>\kappa>0.  
\end{equation}
Moreover, this solution belongs to the space 
\begin{align*}
&\vr_\F \in L^2(0,T;H^3)\cap C(0,T;H^3)\cap H^1(0,T;H^2)\cap C^1(0,T;H^2)\cap H^2(0,T;L^2),\\
&\vu_\F \in L^2(0,T;H^4)\cap C(0,T;H^3)\cap C^1(0,T;H^1)\cap H^2(0,T;L^2),\\
&\vV = \frac{\d\X}{\dt},\vw \in H^2(0,T) \cap C^1([0,T]).
\end{align*}
\end{Theorem}

The advantage of weak solutions is that we usually know they exist without any smallness assumptions. However, their drawback is that they are in general not unique. Therefore two directions are developed to ensure uniqueness of weak solutions. One is to look for some additional assumptions that will guarantee uniqueness. In the case of incompressible Navier-Stokes system it leads to well known Prodi-Serrin-type conditions. Much less is known for compressible flows, we can refer to uniqueness result by Hoff \cite{Hoff} in the nonstationary case and \cite{MP} for stationary inflow problem, both assuming that the pressure is a constant function of the density. Another approach, which is used also in our study, is to show the so called weak-strong uniqueness property, meaning that a weak solution coincides with a strong one as long as the latter exists. The weak-strong uniqueness is based on the  concept of relative entropy (or energy) inequality introduced  already by Dafermos \cite{D} and in case of compressible fluid formulated by Germain \cite{G}. Weak-strong uniqueness for a compressible barotropic flow without additional structural assumptions has been shown in \cite{FeJiNo}. The result has been generalized to complete system with variable temperature in \cite{FN} assuming specific dependence of viscosity and heat conductivity in the temperature. We also refer to \cite{Dobo}, \cite{KrNePi} for problems on moving domains.
Concerning fluid-structure interaction problems, a weak-strong uniqueness result 
for a motion on rigid body in an incompressible fluid has been shown recently in 2D case, see \cite{bravin2018weak}, \cite{GlassSueur} and in 3D case see,  \cite{CNM}, \cite{MNR}. In a case of rigid body with a cavity filled by  incompressible fluid the weak-strong uniqueness was shown in \cite{Disser2016}. For an analogous result for a cavity filled by compressible fluid see \cite{GMN}.

Our goal is to show the weak-strong uniqueness property for system 
\eqref{fluidmotion}-\eqref{bodymotion}. More precisely, our main result reads 
\begin{Theorem} \label{t3}
Let $\bar \vr$ be the mean value of $\vr_{\F 0}$ in $\F_0$.
Assume that $d(\B_0,\de \Omega) > 0$
and the initial conditions $(\vr_{\F 0},\vu_{\F 0},\vV_0,\vw_0)$ satisfy appropriate compatibility conditions, see \cite[formulas (15), (16)]{BG}. Let moreover \eqref{eq:strongass} be satisfied, $\gamma > 3/2$ and let $(\vr_{2 \F},\vu_{2 \F},\X_2,\vw_2)$ be the strong solution to \eqref{fluidmotion}-\eqref{bodymotion} on $(0,T)$ given by Theorem \ref{t2} satisfying \eqref{eq:dist_cond}.

Let $(\vr_1,\vu_1)$ and $\{\overline{\B}_1,\eta_1\}$ be a weak solution to \eqref{fluidmotion}-\eqref{bodymotion} on $(0,T)$ emanating from the same initial data given by Theorem \ref{t1}. Then 
\begin{equation*}
\eta_1[t](x) = \X_2(t) + \tn{O}_2(t)\x,
\end{equation*}
where 
\begin{equation*}
    \frac{\d \tn{O}_2}{\dt}\tn{O}_2^{-1}\y = \vw_2 \times \y, \quad \text{ for all } \y \in \R^d, \qquad \tn{O}_2(0) = \tn{I},
\end{equation*}
for all $t \in (0,T)$, therefore $\B_1(t) = \B_2(t)$, and 
\begin{equation*}
    (\vr_{1 \F},\vu_{1 \F}) = (\vr_{2 \F},\vu_{2 \F})
\end{equation*}
for all $(t,\x) \in Q_{\F 1} = Q_{\F 2}$.
\end{Theorem}
\begin{Remark} 
The smallness assumption \eqref{eq:strongass} is only required for the existence of strong solution, we don't use it in the proof of Theorem \ref{t3}.
\end{Remark}

The relation between the initial data in Theorem \ref{t2} and the initial data in Theorem \ref{t1} is obvious, namely we consider the initial data to be the same if
\begin{equation}
\vr_0  =
\left| 
\begin{array}{c}
\vr_{\mathcal{F} 0} \ \ \mbox{in}\ \Omega \setminus \B_0,\\\
\vr_{\mathcal{B}}(0,\cdot)\ \ \mbox{in}\ \B_0
\end{array}
\right.
\ \ (\vr\vu)_0 =
\left|
\begin{array}{c}
\vr_{\F 0}\vu_{\mathcal{F} 0}\ \ \mbox{in}\ \Omega \setminus \B_0,\\\
\vr_{\mathcal{B}}(0,\x)\left(\vV_0 + \vw_0 \times \x \right)\ \ \mbox{in}\ \B_0.
\end{array}
\right.
\label{eq:ICrelation}
\end{equation}
As already mentioned, an analogous result for a motion of a rigid body in an incompressible fluid has been shown in \cite{CNM}.
The key difficulty in the proof is that in order to compare strong and weak solution we need to have them defined on the same domain. 
However, weak and strong solutions are defined on different time-dependent domains where the motion of the body is given by corresponding velocity. Therefore one has to introduce appropriate change of coordinates to transform the strong solution to the domain of the weak solution. 
This is also the main difference compared to the problem of a body with a cavity studied in \cite{GMN}, where we have the same domain for weak and strong solution. 
The modified strong solution satisfies system \eqref{fluidmotion} with many additional terms coming from the change of coordinates which must be carefully treated.   
\begin{Remark}
The existence of weak solutions formulated in Theorem \ref{t1} holds also for many bodies regardless of possible collisions. However, as the existence of strong solutions (Theorem \ref{t2}) is known only for a single rigid body, for sake of simplicity we restrict our main result formulated in Theorem \ref{t3} to the case of a single body.
\end{Remark}

\section{Relative energy inequality}
\label{rei}

We will deduce the relative energy inequality in the spirit of \cite{FeJiNo} with one adjustment, namely we want to measure the difference of the densities only on the fluid part. Therefore we introduce the following notation. For a weak solution $(\vr,\vu)$ and a pair of test functions $(r,\vU)$ where $r$ is defined on $Q_\F$ and $\vU$ is defined on $Q_T$ we define the relative energy $\mathcal {E}\Big([\vr,\vu]|[r,\vU]\Big)$  as
\begin{align}\label{eq:RE}
\mathcal {E}\Big([\vr,\vu]|[r,\vU]\Big)(\tau)& = \int_{\Omega} \frac 12 \vr\left|\vu-\vU\right|^2 (\tau,\cdot) \dx \\ \nonumber
& + \int_{\F(\tau)}\left( P(\vr) - P'(r)(\vr-r)-P(r) \right)(\tau,\cdot) \dx.
\end{align}
We prove the following 
\begin{Proposition}\label{p:REI}

Let $(\vr,\vu)$ and $\{\overline{\B},\eta\}$ be a weak solution to the problem \eqref{fluidmotion}-\eqref{bodymotion}.  Then $(\vr,\vu)$ satisfies the following relative energy inequality
\begin{align}\label{eq:REI}
&\mathcal {E}\Big([\vr,\vu]|[r,\vU]\Big)(\tau) + \int_0^\tau\int_{\F(t)} \left(\tn{S}(\Grad \vu) - \tn{S}(\Grad \vU)\right):\left(\Grad \vu - \Grad \vU\right)\dxdt \\ \nonumber \leq\, &\mathcal {E}\Big([\vr_0,\vu_0]|[r(0,\cdot),\vU(0,\cdot)]\Big) + \int_0^\tau \mathcal{R}(\vr,\vu,r,\vU)(t) \dt
\end{align}
for a.a. $\tau \in (0,T)$ and any pair of test functions $(r,\vU)$ such that $r \in C^{\infty}(\overline{Q_\F})$, $r > 0$, $\vU \in C^{\infty}(\overline{Q_T})$, $\vU = 0$ on $\partial \Omega$ and 
 \begin{equation} \label{dzero}
 \tn{D}(\vU)=0 \; \textrm{in a neighbourhood of} \; \B(t).    
 \end{equation}
The remainder term $\mathcal{R}$ is given by
\begin{align}\label{eq:R}
&\mathcal{R}(\vr,\vu,r,\vU)(t) = \int_{\Om} \vr(\partial_t\vU + \vu\cdot\Grad\vU)\cdot(\vU-\vu) \dx \\ \nonumber
& \quad + \int_{\F(t)} \tn{S}(\Grad\vU):(\Grad\vU-\Grad\vu) + \Div\vU (p(r) - p(\vr)) + (r-\vr)\partial_t P'(r) \dx \\ \nonumber
&\quad + \int_{\F(t)} (r\vU - \vr\vu)\cdot\Grad P'(r)\dx + \int_{\de \B(t)} p(r) (\vu-\vU)\cdot \vn\, \d S.
\end{align}
\end{Proposition}

\begin{proof}
We start with the momentum equation \eqref{moment} where we use $\vU$ as a test function\footnote{More precisely we use $\vU\psi^\ep$ as a test function, where $\psi(t)$ is a characteristic function of the interval $[0,\tau]$ and $\psi^\ep$ is its mollification around the point $t = \tau$. Then passing with $\ep$ to zero we recover \eqref{moment2}.}. We get
\begin{align}\label{moment2}
&\int_0^\tau \int_{\Omega}(\vr \vu)\cdot \partial_t
\vU + [\vr \vu\otimes \vu]: \tn{D}(\vU)
+p\mbox{div}_x\vU - \tn{S}(\nabla_x\vu): \tn{D}(\vU) \dxdt \\ \nonumber
& \quad = \int_\Omega (\vr\vu\cdot\vU)(\tau,\cdot)\dx  -\int_\Omega (\vr\vu)_0\cdot\vU(0,\cdot) \dx
\end{align}
Next we use the continuity equation \eqref{continu} with test function $\frac{1}{2}|\vU|^2$ to obtain
\begin{equation}
\int_0^\tau \int_{\Omega}\vr\vU\cdot \partial_t \vU + \vr
(\vu\cdot\nabla_x \vU) \cdot \vU \dx \dt = \int_\Omega \frac 12 \vr|\vU|^2(\tau,\cdot) \dx - \int_\Omega\frac 12 \vr_0|\vU|^2(0,\cdot) \dx. \label{continu2}
\end{equation}
Summing the energy inequality \eqref{energy} with \eqref{continu2} and subtracting \eqref{moment2} we end up with
\begin{align}\label{eq:REIv1}
    &\int_\Omega \left(\frac 12 \vr |\vu-\vU|^2 + P(\vr)\right)(\tau,\cdot) \dx + \int_0^\tau\int_\Omega (\tn{S}(\nabla_x\vu) - \tn{S}(\nabla_x\vU)):\tn{D}(\vu-\vU)\dxdt \\ \nonumber
    &\quad \leq \int_\Omega \left( \frac 12 \frac{|(\vr\vu)_0-\vr_0\vU(0)|^2}{\vr_0} + P(\vr_0)\right) \dx \\ \nonumber
    & + \int_0^\tau \int_\Omega \vr(\de_t \vU + \vu\cdot\nabla_x\vU)\cdot (\vU-\vu) + \tn{S}(\nabla_x\vU):\tn{D}(\vU-\vu) - p(\vr)\Div \vU \dxdt.
\end{align}
It is easy to observe that 
\begin{equation}\label{eq:triv}
    \int_{\B(\tau)} P(\vr)(\tau,\cdot)\dx = \int_{\B_0} P(\vr_0) \dx.
\end{equation}
Next, we observe that the weak formulation of the continuity equation can actually be split into the fluid part and the solid part. For this purpose we recall that on the body we have \eqref{eq:CE_body} and therefore using the transport theorem which
we recall in the Appendix, see \eqref{trans}, we have
\begin{equation*}
\int_0^\tau \int_{\B(t)}\vr_\B \partial_t \phi + \vr_\B
\vu_\B\cdot\nabla_x \phi \dx \dt = \int_{\B(\tau)} \vr_\B\phi(\tau,\cdot)\dx - \int_{\B_0} \vr_B(0)\phi(0,\cdot) \dx.
\end{equation*}
This means that we have also the same equation satisfied on the fluid part
\begin{equation}\label{eq:CEfluid}
\int_0^\tau \int_{\F(t)}\vr_\F \partial_t \phi + \vr_\F
\vu_\F\cdot\nabla_x \phi \dx \dt = \int_{\F(\tau)} \vr_\F\phi(\tau,\cdot)\dx - \int_{\F_0} \vr_F(0)\phi(0,\cdot) \dx.
\end{equation}
We use $P'(r)$ as a test function in the continuity equation \eqref{eq:CEfluid}
\begin{equation}\label{eq:CEfluid2}
\int_0^\tau \int_{\F(t)}\vr \partial_t P'(r) + \vr
\vu\cdot\nabla_x P'(r) \dx \dt = \int_{\F(\tau)} \vr P'(r)(\tau,\cdot)\dx - \int_{\F_0} \vr_0 P'(r)(0,\cdot) \dx.
\end{equation}
Moreover we use the transport theorem once again to get
\begin{equation}\label{eq:trans100}
\int_0^\tau \int_{\F(t)} \partial_t p(r) \dx \dt + \int_0^\tau \int_{\de \B(t)} p(r)\vu\cdot\vn \,\d S \dt = \int_{\F(\tau)} p(r)(\tau,\cdot)\dx - \int_{\F_0} p(r)(0,\cdot) \dx.
\end{equation}
Finally, we have
\begin{equation}\label{eq:trans101}
    \int_0^\tau \int_{\F(t)} r\vU\cdot\nabla_x P'(r) + p(r)\Div \vU \dxdt = \int_0^\tau \int_{\de \B(t)} p(r)\vU\cdot\vn\,\d S \dt .
\end{equation}
We combine \eqref{eq:REIv1}, \eqref{eq:CEfluid2}, \eqref{eq:trans100} and \eqref{eq:trans101} and realize that $\tn{D}(\vu) = \tn{D}(\vU) = 0$ on $\B(t)$ and therefore also $\Div \vu = \Div \vU = 0$ on $\B(t)$. This yields the relative entropy inequality \eqref{eq:REI} with the remainder term \eqref{eq:R}.
\end{proof}

\section{Change of coordinates}\label{s:cc}

Suppose that $(\vr_1,\vu_1)$, $\{\eta_1,\overline{\B}_1\}$ is a weak solution to \eqref{fluidmotion}-\eqref{bodymotion} and $(\vr_{2 \F},\vu_{2 \F},\X_2,\vw_2)$ is a strong solution. We denote by $\B_1(t)$ and $\F_1(t)$ the domains occupied by, respectively, solid and fluid at time $t$ for the weak solution and, analogously, $\B_2(t)$ and $\F_2(t)$ for the strong solution. We define $\vV_2 = \frac{\d}{\dt}\X_2$ and
\begin{equation}
\vu_2 =
\left|
\begin{array}{c}
\vu_{2 \mathcal{F}}\ \ \mbox{in}\ Q_{\F 2},\\\
\vV_2 + \vw_2\times(\x-\X_2)\ \ \mbox{in}\ Q_{\B 2}.
\end{array}
\right.
\end{equation}

There is no reason a priori, why the body $\B_1(t)$ of the weak solution would not touch the boundary $\de\Omega$ at some time instant $t < T$. Therefore we introduce time $T_{min}$ such that
\begin{equation} \label{eq:tmin}
    T_{min} = \inf \left\{t \in (0,T); d(\B_1(t),\de\Omega) \leq \frac{\kappa}{2}\right\}, 
\end{equation}
where $\kappa$ is given by the property \eqref{eq:dist_cond} of the strong solution. In particular we have $T_{min} > 0$ and on the interval $(0,T_{min})$ there is no collision between the body $\B_1$ and the boundary $\de \Omega$. We perform our further analysis on the time interval $(0,T_{min})$, however in order to simplify the notation we keep denoting the time interval by $(0,T)$. We will return to the notation $T_{min}$ at the very end of the proof of Theorem \ref{t3} in the final paragraph of Section \ref{s:proof}.

 Our goal here is to transform the strong solution to the domain of the weak solution in such a way that $\B_2(t)$ is transformed to $\B_1(t)$. For this purpose we introduce cutoff functions $\zeta_i(t,\x)$ for $i=1,2$ such that $\zeta_i(t,\x)=1$ in a neighborhood of $\B_i(t)$ and $\zeta_i(t,\x)=0$ in a neighborhood of $\de \Omega$. Moreover we extend the domain of definition of the rigid body motions $\vu_{i\B}$ to the whole $(0,T)\times \Omega$, i.e. for $\x \in \Omega$ we have
\begin{equation*}
\vu_{i \mathcal{B}}(t,\x)=\vV_i(t) + \vw_i(t)\times (\x-\X_i(t)).    
\end{equation*}
Then we set 
\begin{equation*}
\Lambda_i(t,\x)=\zeta_i(t,\x)\vu_{i\B}(t,\x).     
\end{equation*}
Notice that $\Lambda_i$ is a rigid motion in a neighborhood of $\B_i(t)$ and vanishes in a vicinity of $\de \Omega$. 
Now we define transformations 
$$
\bZ_i:\Om \to \Om
$$ 
as solutions to ODEs 
\begin{align*}
&\frac{\d}{\d t}\bZ_i(t,\y)=\Lambda_i(t,\bZ_i(t,\y)) \quad \forall \y \in \Om, \; t \in (0,T),\\    
&\bZ_i(0,\y)=\y.
\end{align*}
Then we introduce $\bY_i=\bZ_i^{-1}$ and define mappings $\widetilde{\bZ}_i:\Om \to \Om$ as 
\begin{align}
\widetilde{\bZ}_1(t,\x)&=\bZ_1(t,\bY_2(t,\x)), \\
\widetilde{\bZ}_2(t,\x)&=\bZ_2(t,\bY_1(t,\x)). 
\end{align}
Note that $\widetilde{\bZ}_1(t,\cdot) = \widetilde{\bZ}_2^{-1}(t,\cdot)$ and $\widetilde{\bZ}_2(t,\B_1(t)) = \B_2(t)$ for all $t \in [0,T)$. In particular in the neighborhoods of the body $\B_{2}(t)$ the mapping $\widetilde{\bZ}_1(t,\cdot)$ is rigid and vice versa, more precisely we have
\begin{align}
\widetilde{\bZ}_1(t,\x) &= \X_1(t) + \tn{O}_1(t)\tn{O}_2^T(t)(\x - \X_2(t)) \qquad \text{ in the neighborhood of } \B_2(t), \\
\widetilde{\bZ}_2(t,\x) &= \X_2(t) + \tn{O}_2(t)\tn{O}_1^T(t)(\x - \X_1(t)) \qquad \text{ in the neighborhood of } \B_1(t). \label{eq:bzexpress}
\end{align}
For simplicity of notation we denote 
\begin{equation} \label{def:tildeo}
\Ot(t) = \tn{O}_2(t)\tn{O}_1^T(t).    
\end{equation}
We define now the transformed strong solution $(r^s,\bU^s)$ as 
\begin{align} \label{newsol1}
r^s(t,\x) &=\vr_{2 \F}(t,\widetilde{\bZ}_2(t,\x)) \qquad \qquad \qquad \quad \text{ for all } (t,\x) \in Q_{\F 1}, \\ \label{newsol2}
\bU^s(t,\x) &=\tn{J}_{\widetilde{\bZ}_1}(t,\widetilde{\bZ}_2(t,\x))\vu_2(t,\widetilde{\bZ}_2(t,\x)) \quad \text{ for all } (t,\x) \in Q_T,
\end{align}
where $(\tn{J}_{\widetilde{\bZ}_1})_{ij}(t,\widetilde{\bZ}_2(t,\x)) = \frac{\de (\widetilde{\bZ}_1)_i}{\de \x_j}(t,\widetilde{\bZ}_2(t,\x))$. 
In particular we have for $\x \in \B_1(t)$
\begin{equation} \label{newsol3}
    \bU^s(t,\x) = \vV^s(t) + \vw^s(t) \times (\x-\X_1(t))
\end{equation}
with
\begin{equation}\label{eq:strongbody}
    \vV^s(t) = \Ot^T(t)\vV_2(t), \qquad \qquad \vw^s(t) = \Ot^T(t)\vw_2(t).
\end{equation}

\begin{Lemma}\label{l:31}
\begin{itemize}
    \item[(i)] It holds
    \begin{equation}
        \Ot^T(t)\frac{\d \Ot}{\dt}(t)\x = (\vw^s-\vw_1)(t)\times \x.
    \end{equation}
    \item[(ii)] The following estimates hold for $\x \in \de \B_1(t)$, $t \in (0,T)$
    \begin{align}\label{eq:315}
        |\widetilde{\bZ}_2(t,\x) - \x| &\leq C\left(\|\vV_1-\vV^s\|_{L^2(0,t)} + \|\vw_1-\vw^s\|_{L^2(0,t)}\right),\\ \label{eq:316}
        |\de_t \widetilde{\bZ}_2(t,\x)| &\leq C\left(|\vV_1-\vV^s|(t) + |\vw_1-\vw^s|(t)\right).
    \end{align}
    \item[(iii)] The following estimates hold for $t \in (0,T)$
    \begin{align}\label{eq:317}
        \|\widetilde{\bZ}_2(t,\cdot) - \mathrm{id}\|_{W^{3,\infty}(\F_1(t))} &\leq C\left(\|\vV_1-\vV^s\|_{L^2(0,t)} + \|\vw_1-\vw^s\|_{L^2(0,t)}\right),\\
        \|\de_t \widetilde{\bZ}_2(t,\cdot)\|_{W^{1,\infty}(\F_1(t))} &\leq C\left(|\vV_1-\vV^s|(t) + |\vw_1-\vw^s|(t)\right). \label{eq:318}
    \end{align}
\end{itemize}
\end{Lemma}
\begin{proof}
\begin{enumerate}
    \item[(i)] We calculate
    \begin{align}
        \frac{\d \Ot}{\dt} &= \frac{\d \tn{O}_2}{\dt}\tn{O}_1^T + \tn{O}_2\frac{\d \tn{O}_1^T}{\dt} = \tn{O}_2\tn{O}_1^T\tn{O}_1\tn{O}_2^T\frac{\d \tn{O}_2}{\dt}\tn{O}_1^T - \tn{O}_2\tn{O}_1^T\frac{\d \tn{O}_1}{\dt}\tn{O}_1^T \\
        &=\Ot\left(\tn{O}_1\tn{O}_2^T\frac{\d \tn{O}_2}{\dt}\tn{O}_1^T - \frac{\d \tn{O}_1}{\dt}\tn{O}_1^T\right) \nonumber
    \end{align}
    and therefore 
    \begin{equation}
        \Ot^T\frac{\d \Ot}{\dt}\x = \left(\tn{O}_1\tn{O}_2^T\frac{\d \tn{O}_2}{\dt}\tn{O}_1^T - \frac{\d \tn{O}_1}{\dt}\tn{O}_1^T\right)\x.
    \end{equation}
    Since by definition it is $\frac{\d \tn{O}_1}{\dt}\tn{O}_1^T\x = \vw_1 \times \x$, it remains to show that $\tn{O}_1\tn{O}_2^T\frac{\d \tn{O}_2}{\dt}\tn{O}_1^T\x = \vw^s \times \x$.
    For this purpose we notice that for any rotation matrix $\mathbb{R}$ and vectors $\mathbf{a}, \mathbf{b}$ we have   
    \begin{equation} \label{eq:rot}
    \mathbb{R}\mathbf{a} \times \mathbb{R}\mathbf{b} = \mathbb{R}(\mathbf{a}\times \mathbf{b}).
    \end{equation}
    Therefore, using the definition of $\vw_s$ and $\vw_2$
    \begin{equation}
        \vw^s \times \x = \Ot^T\vw_2 \times \Ot^T \Ot\x = \tn{O}_1\tn{O}_2^T(\vw_2 \times \tn{O}_2\tn{O}_1^T\x) = \tn{O}_1\tn{O}_2^T\frac{\d \tn{O}_2}{\dt}\tn{O}_2^T\tn{O}_2\tn{O}_1^T\x.
    \end{equation}
    \item[(ii)]
     We derive \eqref{eq:bzexpress} with respect to time to get
\begin{align}\label{eq:detZ}
    \de_t\widetilde{\bZ}_2(t,\x) &= \vV_2(t) + \frac{\d \Ot}{\dt}(t)(\x - \X_1(t)) - \Ot(t)\vV_1(t) \\
    &= \Ot(t)\left(\vV^s(t) - \vV_1(t) + \Ot^T(t)\frac{\d \Ot}{\dt}(t)(\x - \X_1(t))\right) \nonumber \\ \nonumber
    &= \Ot(t)\big(\vV^s(t) - \vV_1(t) + (\vw^s-\vw_1)(t)\times(\x - \X_1(t))\big),
\end{align}
which proves \eqref{eq:316}, because both $\X_1$ and $\Ot$ are uniformly bounded. In order to prove \eqref{eq:315} we integrate \eqref{eq:detZ} with respect to time and obtain
\begin{equation} \label{eq:Z2}
    \left(\widetilde{\bZ}_2 - \mathrm{id}\right)(t,\x) = \int_0^t \Ot(\tau)\big(\vV^s(\tau) - \vV_1(\tau) + (\vw^s-\vw_1)(\tau)\times(\x - \X_1(\tau))\big) \d \tau.
\end{equation}
\item[(iii)] The estimates \eqref{eq:317} and \eqref{eq:318} follow from \eqref{eq:315} and \eqref{eq:316} by means of a standard construction of change of variables similar as is described in \cite[Proposition 6.1]{GGH13}.

\end{enumerate}
\end{proof}

\begin{Remark}
 Let us mention that such type of transformation was introduced by Inoue, Wakimoto \cite{inoue1977existence} and properties of such transformation was described in details by Takahashi, see \cite{T}.
 
 The same transformation has been used in \cite{CNM,MNR} to show the weak-strong uniqueness property for the incompressible fluid-structure interaction problem. Although when working with compressible models one usually does not need to define the transformed velocity as multiplied by the Jacobian of the change of coordinates, since there is no need to keep the velocity divergence-free, Lemma \ref{l:31} reveals the importance of using this definition in our framework.

\end{Remark}


The transformed strong solution satisfies the following system of equations in the fluid part of the domain $\Omega$
\begin{equation}\label{eq:fluidtransformed}
\left\{
\begin{array}{ccc}
\de_t r^s + \dv_{x}(r^s\bU^s)=G(r^s,\bU^s) & {\rm in} & Q_{\F 1},\\
\de_t(r^s \bU^s) + \dv_{x}(r^s \bU^s \otimes \bU^s)-\dv_{x}\tn{S}(\nabla_x \vU^s)
+\nabla_{x}p(r^s) = \mathbf{F}(r^s,\bU^s) & {\rm in} & Q_{\F 1}, \\

\bU^s=0 & {\rm on} & (0,T) \times \de \Omega,\\
r^s(0,\cdot)=\vr_{\F 0} & {\rm in} & \Om \setminus \B_0, \\
\bU^s(0,\cdot)= \vu_{\F 0} & {\rm in} & \Om\setminus \B_0.
\end{array}
\right.
\end{equation}

In order to express the terms on the right hand sides of the transformed continuity and momentum equations we introduce the following notation, where we use the summation convention.
\begin{align} \label{eq:Hdef}
    \tn{H} &= (\nabla_x\ZZ)^{-1} = \nabla_x \ZZj, \quad \text{ i.e. } (\tn{H})_{ij} = \de_{j} (\ZZj)_i,\\ \label{eq:GGdef}
    \tn{G} &= \tn{H}\tn{H}^T, \quad \text{ i.e. } (\tn{G})_{ij} = \de_{k} (\ZZj)_i \de_{k} (\ZZj)_j, \\
    \Gamma^i_{\alpha \beta} &= \de_{l} (\ZZj)_i \de_{\alpha\beta} (\ZZ)_l. \label{eq:Gammadef}
\end{align}
Using these expressions we have
\begin{align} \label{eq:Gdef}
G(r^s,\vU^s) &= 
\tn{H}_{jk}\de_t(\ZZ)_k \de_j r^s, \\ \label{eq:Fdef}
\mathbf{F}_i(r^s,\vU^s) &= -r^s\tn{H}_{ij}\de_t\de_\beta(\ZZ)_j\vU^s_\beta + r^s\Gamma^i_{\alpha\beta}\tn{H}_{\alpha j} (\de_t \ZZ)_j \vU^s_\beta + \tn{H}_{\alpha j} (\de_t \ZZ)_j \de_\alpha (r^s\vU^s_i) \\ \nonumber
&\quad - r^s\Gamma^i_{\alpha\beta} \vU^s_\alpha\vU^s_\beta + (\mu+\lambda)(\tn{G}-\tn{I})_{i\alpha}\de_\alpha \dv_{x}\vU^s - (\tn{G}-\tn{I})_{i\alpha}\de_\alpha p(r^s) \\ \nonumber
&\quad +\mu((\tn{G}-\tn{I})_{\alpha\beta}\de_{\alpha\beta}\vU^s_i + \Delta_x (\ZZj)_\beta \de_\beta \vU^s_i + 2\Gamma^i_{\alpha\gamma}\tn{G}_{\alpha\beta}\de_\beta \vU^s_\gamma) \\ \nonumber
&\quad + \mu(\Gamma^i_{\alpha\gamma}\Delta_x (\ZZj)_\alpha\vU^s_\gamma + \tn{H}_{ij}\tn{G}_{\alpha\beta}\de_{\alpha\beta\gamma}(\ZZ)_j\vU^s_\gamma).
\end{align}
Expressing the equations on the rigid body we have
\begin{align} \label{eq:rigidtransformed1}
    m\frac{\d \vV^s}{\dt} &= -m (\vw^s-\vw_1)\times \vV^s - \int_{\partial\mathcal{B}_1(t)} \left(\tn{S}(\nabla_x\vU^s) - p(r^s)\tn{I}\right)\vc{n}\, \d S\quad
 \text{ in } (0,T), \\ \label{eq:rigidtransformed2}
 \tn{J}_1\frac{\d \vw^s}{\dt} &= \tn{J}_1\vw^s\times\vw^s - \tn{J}_1\left((\vw^s-\vw_1)\times \vw^s\right) \\ \nonumber &\quad - \int_{\partial\mathcal{B}_1(t)} (\x-\X_1(t)) \times \left(\tn{S}(\nabla_x\vU^s) - p(r^s)\tn{I}\right)\vc{n}\, \d S\quad \text{ in } (0,T),\\ \label{eq:rigidtransformed3}
 \vV^s(0) &= \vV_{0},\qquad \vw^s(0)= \vw_{0}, 
\end{align}
 where we used $\tn{J}_1 = \Ot^T \tn{J}_2\Ot$. The formulas \eqref{eq:Hdef}-\eqref{eq:rigidtransformed2} are obtained by lengthy although direct calculations similar to the proof of Lemma 3.1 from \cite{CNM}.

\section{Proof of Theorem \ref{t3}} \label{s:proof}
In order to show the weak-strong uniqueness property we would like 
to apply the relative energy inequality \eqref{eq:REI} with $r = r^s$ and $\bU=\bU^s$.
However, recall that $\bU$ in the energy inequality must in particular satisfy the condition \eqref{dzero}, whereas $\bU^s$ is a rigid motion only on $\B_1(t)$. Therefore we have to modify it slightly. For this purpose for given $\ep>0$ we introduce a smooth cutoff function $\xi^{\ep}:Q_T \to Q_T$ such that $\xi^{\ep}(t,\x) \in [0,1]$ and 
\begin{equation}
\xi^{\ep}(t,\x)=\left\{ \begin{array}{l}
1, \quad \x \in [\B_1(t)]_{\ep},\\
0, \quad \x \in \Om \setminus [\B_1(t)]_{2\ep}.
\end{array}
\right.
\end{equation}
Now we define 
\begin{equation}
\bU^s_{\ep}=\xi_{\ep}\bU^s_{\B}+(1-\xi_{\ep})\bU^s_{\F}.    
\end{equation}
and we have 
\begin{align} \label{prop_uep}
\tn{D}(\bU^s_{\ep})=0 \; {\rm on} \; [\B_1(t)]_{\ep}, \quad 
\bU^s_{\ep}=\bU^s \; {\rm on} \; \B_1(t) \cup (\Om \setminus [\B_1(t)]_{2\ep}), 
\end{align}
where we denote by $X_\ep$ the $\ep$-neighborhood of the set $X$.

Now we apply the relative energy inequality \eqref{eq:REI} with $(\vr,\vu) = (\vr_1,\vu_1)$ being the weak solution and $(r,\vU) = (r^s,\bU^{s}_{\ep})$ being the test functions. In what follows, to keep the notation simple, we will use $(\vr,\vu)$ instead of $(\vr_1,\vu_1)$, $(r,\vU_\ep)$ instead of $(r^s,\vU_\ep^s)$ and $\B(t), \F(t)$ instead of $\B_1(t), \F_1(t)$. Let us first examine the left hand side of the relative energy inequality \eqref{eq:REI}. We define 
\begin{equation}
\Gamma^\ep(t) =  [\B(t)]_{2\ep} \setminus \B(t),  \quad
\F^\ep(t) = \Om\setminus [\B(t)]_{2\ep}.
\end{equation}
We write 
\begin{equation*}
     \int_\Omega \frac 12 \vr\left|\vu-\vU_\ep\right|^2 (\tau,\cdot) \dx = \int_\Omega \frac 12 \vr\left(\left|\vu-\vU\right|^2 + 2(\vu-\vU)\cdot(\vU-\vU_\ep) + |\vU-\vU_\ep|^2\right) (\tau,\cdot)\dx
\end{equation*}
and observing that $\vU \neq \vU_\ep$ only in the set $\Gamma^\ep$ we show that
\begin{equation}\label{eq:UUep}
    \lim_{\ep\rightarrow 0} \int_{\Gamma^\ep(\tau)} \left(\vr(\vu-\vU)\cdot(\vU-\vU_\ep) +  \frac 12\vr \left|\vU-\vU_\ep\right|^2\right) (\tau,\cdot) \dx = 0.
\end{equation}
Indeed, we have
\begin{align*}
&\left|\int_{\Gamma^\ep(\tau)} \vr(\vu-\vU)\cdot(\vU-\vU_\ep)(\tau,\cdot) \dx\right| \leq  \\  
&\qquad \qquad |\Gamma^\ep(\tau)|^\beta \left\|\vr|\vu-\vU|^2\right\|_{L^1(\Gamma^\ep(\tau))} \left\|\vr\right\|_{L^\gamma(\Gamma^\ep(\tau))} \left\|\vU-\vU_\ep \right\|_{L^\infty(\Gamma^\ep(\tau))}
\end{align*}
for some $\beta > 0$ and since all norms on the right hand side are bounded, the whole integral converges to zero. Similarly we handle the second term in \eqref{eq:UUep}.

On the rigid body  $\B(t)$ we have 
\begin{align*}
    \vu_{\mathcal{B}}(t,\x) &= \vV_1(t) + \vw_1(t)\times (\x-\X_1(t)), \\
    \vU_{\ep \B}(t,\x) = \vU_{\mathcal{B}}(t,\x) &= \vV^s(t) + \vw^s(t)\times (\x-\X_1(t)),    
\end{align*}
and therefore by \eqref{def:m} and \eqref{def:J}
\begin{equation*}
    \int_{\B(\tau)} \frac 12 \vr\left|\vu-\vU\right|^2 (\tau,\cdot) \dx = \frac{m}{2} |\vV_1-\vV^s|^2(\tau) + \frac{1}{2} \tn{J}_1(\tau)(\vw_1-\vw^s)\cdot(\vw_1-\vw^s)(\tau).
\end{equation*}
One can also observe that there exists a constant $c > 0$ independent of $t$ such that
\begin{equation}
    c|\vw|^2 \leq \tn{J}_1(t)\vw\cdot\vw.
\end{equation}
In the limit $\ep \rightarrow 0$ we therefore recover 
the following lower bound on the left hand side of the relative energy inequality \eqref{eq:REI} 
\begin{align}\label{eq:REIleft}
&\mathcal {E}\Big([\vr,\vu]|[r,\vU]\Big)(\tau) + \int_0^\tau\int_{\F(t)} \left(\tn{S}(\Grad \vu) - \tn{S}(\Grad \vU)\right):\left(\Grad \vu - \Grad \vU\right)\dxdt \geq \\ \nonumber
    &C\left[ \int_{\F(\tau)} \left(\frac 12 \vr\left|\vu-\vU\right|^2  + P(\vr,r)\right)(\tau,\cdot) \dx + |\vV_1-\vV^s|^2(\tau) + |\vw_1-\vw^s|^2(\tau)\right. \\ \nonumber
    &\left. \quad + \int_0^t \int_{\F(t)} (\tn{S}(\nabla_x\vu) - \tn{S}(\nabla_x \vU)):\tn{D}(\vu-\vU) \dxdt \right], \\
\end{align}
with the notation
\begin{equation}
    P(\vr,r) = P(\vr) - P'(r)(\vr-r) - P(r).
\end{equation}
Let us now examine the right hand side of \eqref{eq:REI}. First of all, it is easy to observe that 
\begin{equation*}
\lim_{\ep \rightarrow 0} \mathcal {E}\Big([\vr_0,\vu_0]|[r(0,\cdot),\vU_\ep(0,\cdot)]\Big) = 0,
\end{equation*}
and we have also
\begin{equation*}
\lim_{\ep \rightarrow 0} \int_0^\tau \mathcal{R}(\vr,\vu,r,\vU_\ep) \dt = \int_0^\tau \mathcal{R}(\vr,\vu,r,\vU) \dt.
\end{equation*}
Summing up, taking $\ep \to 0$ we obtain from \eqref{eq:REIleft} and \eqref{eq:e0}
\begin{align}\label{eq:REIleft2}
    &\int_{\F(\tau)} \left(\frac 12 \vr\left|\vu-\vU\right|^2  + P(\vr,r)\right)(\tau,\cdot) \dx + |\vV_1-\vV^s|^2(\tau) + |\vw_1-\vw^s|^2(\tau) \\ \nonumber
    &\quad + \int_0^{\tau} \int_{\F(t)} (\tn{S}(\nabla_x\vu) - \tn{S}(\nabla_x \vU)):\tn{D}(\vu-\vU) \dxdt \\ \nonumber
    &\leq 
    C\left[\mathcal {E}\Big([\vr_0,\vu_0]|[r(0,\cdot),\vU(0,\cdot)]\Big) + \int_0^\tau \mathcal{R}(\vr,\vu,r,\vU)(t) \dt\right].
\end{align}
Furthermore, by \eqref{newsol3} we have 
\begin{equation*} 
\int_{\B(\tau)}\vr|\vu-\vU|^2(\tau,\cdot)\dx \leq C [|\vV_1-\vV^s|^2(\tau)+|\vw_1-\vw^s|^2(\tau)].    
\end{equation*}
Combining this inequality with \eqref{eq:RE} and \eqref{eq:REIleft2} 
we see that 
\begin{equation} \label{eq:e0}
\mathcal{E}\Big([\vr,\vu]|[r,\vU]\Big)(\tau) \leq 
    C\left[\mathcal {E}\Big([\vr_0,\vu_0]|[r(0,\cdot),\vU(0,\cdot)]\Big) + \int_0^\tau \mathcal{R}(\vr,\vu,r,\vU)(t) \dt\right].    
\end{equation}
We therefore analyze the remainder term
\begin{align}\label{eq:R2}
&\mathcal{R}(\vr,\vu,r,\vU)(t) = \int_{\F(t)} \vr(\partial_t\vU + \vu\cdot\Grad\vU)\cdot(\vU-\vu) \dx + \int_{\B(t)} \vr(\partial_t\vU + \vu\cdot\Grad\vU)\cdot(\vU-\vu) \dx \\ \nonumber
& \quad + \int_{\F(t)} \left( \tn{S}(\Grad\vU):(\Grad\vU-\Grad\vu) + \Div\vU (p(r) - p(\vr)) + (r-\vr)\partial_t P'(r)\right) \dx \\ \nonumber
&\quad + \int_{\F(t)} (r\vU - \vr\vu)\cdot\Grad P'(r)\dx + \int_{\de \B(t)} p(r) (\vu-\vU)\cdot \vn\, \d S = \mathcal{I}_1^F(t)+\mathcal{I}_1^B(t)+ \sum_{i=2}^6 \mathcal{I}_i(t).
\end{align}
We start with estimating $\mathcal{I}_1^F$. For this purpose we use the identity
\begin{equation}\label{eq:trick}
    \vu\cdot\nabla_x\vU\cdot(\vU-\vu) = \vU\cdot\nabla_x\vU\cdot(\vU-\vu) + (\vu-\vU)\cdot\nabla_x\vU\cdot(\vU-\vu).
\end{equation}
Since $\nabla_x\vU$ is bounded, we get 
\begin{equation}\label{eq:trick2}
\left|\int_0^\tau \int_{\F(t)} \vr (\vu-\vU)\cdot\nabla_x\vU\cdot(\vU-\vu) \dxdt\right| \leq \int_0^\tau h_1(t) \mathcal{E}([\vr,\vu]|[r,\vU]) \dt
\end{equation}
for some nonnegative $h_1 \in L^1(0,T)$. Analogously, in the remainder of this section $h_k(t)$ will denote a nonnegative integrable function.

Next we use the strong formulation of the equations \eqref{eq:fluidtransformed} satisfied by $(r,\vU)$ in $\F(t)$, namely
\begin{equation*}
    \de_t\bU+\bU\cdot\nabla_{x}\bU = \frac{1}{r}\left(\Div \tn{S}(\nabla_x\vU) - \nabla_x p(r) + \mathbf{F}(r,\bU) - G(r,\bU)\bU\right).
\end{equation*}
By this way we have
\begin{align}\label{eq:bubu0}
    &\int_0^\tau \int_{\F(t)} \vr(\de_t\vU + \vU\cdot\nabla_x\vU)\cdot(\vU-\vu) \dxdt \\ \nonumber
    & \quad = \int_0^\tau \int_{\F(t)} \frac{\vr}{r}\left( \Div \tn{S}(\nabla_x\vU) - \nabla_x p(r) + \mathbf{F}(r,\bU) - G(r,\bU)\bU\right)\cdot(\vU-\vu) \dxdt = \sum_{i=1}^4 \mathcal{J}_i.
\end{align}
In order to handle $\mathcal{J}_1$ we write
\begin{equation}\label{eq:bubu1}
    \frac{\vr}{r} \Div \tn{S}(\nabla_x\vU) \cdot (\vU-\vu) = \frac{\vr-r}{r}\Div \tn{S}(\nabla_x\vU) \cdot (\vU-\vu) + \Div \tn{S}(\nabla_x\vU) \cdot (\vU-\vu),
\end{equation}
and combining the last term in \eqref{eq:bubu1} with $\int_0^\tau \mathcal{I}_2(t)\dt$ in \eqref{eq:R2} we get
\begin{align*}
    &\int_0^\tau\int_{\F(t)} \Div \tn{S}(\nabla_x\vU) \cdot (\vU-\vu) + \tn{S}(\nabla_x\vU) : \nabla_x (\vU-\vu) \dxdt \\ \nonumber
    & = \int_0^\tau\int_{\F(t)} \Div(\tn{S}(\nabla_x\vU)(\vU-\vu)) \dxdt = \int_0^\tau \int_{\de \B(t)} \tn{S}(\nabla_x\vU)\vn\cdot(\vU-\vu)\,\d S \dt = \mathcal{J}_5.
\end{align*}
Moreover we estimate
\begin{equation}\label{eq:DDD}
    \int_0^\tau\int_{\F(t)} \frac{\vr-r}{r}\Div \tn{S}(\nabla_x\vU) \cdot (\vU-\vu) \dxdt
\end{equation}
in the same way as was originally done in \cite{FeJiNo}, for more details see also \cite[Section 6.3]{KrNePi}, where the authors study the moving domain case relevant to \eqref{eq:DDD}. The integral \eqref{eq:DDD} is decomposed into three parts and the case $\vr$ small, $\vr$ comparable to $r$ and $\vr$ large are estimated separately, moreover the estimate in the last case is a little bit different if $\gamma < 2$ and if $\gamma \geq 2$. We use a similar procedure also later in our proof, when we estimate the terms containing $\vc{F}(r,\vU)$ and $G(\vr,\vU)$, see \eqref{eq:starterrorterms} and following estimates. We obtain
\begin{align}\label{eq:DDD2}
    &\left|\int_0^\tau\int_{\F(t)} \frac{\vr-r}{r}\Div \tn{S}(\nabla_x\vU) \cdot (\vU-\vu) \dxdt\right| \\ \nonumber
    &\quad \leq \delta \left\|\tn{S}(\nabla_x\bU-\nabla_x\vu)\right\|^2_{L^2(Q_\F)} + C(\delta)\int_0^\tau h_2(t)\mathcal{E}([\vr,\vu]|[r,\vU]) \dt.
\end{align}
%
Next we combine $\mathcal{J}_2$ with $\mathcal{I}_5$ and use \eqref{prespot} to get
\begin{align}\label{eq:bubu100}
    &\int_0^\tau\int_{\F(t)} (r\vU-\vr\vu)\cdot\nabla_x P'(r) - \frac{\vr}{r}\nabla_x p(r)\cdot(\vU-\vu)\dxdt \\ \nonumber
    & \quad =  \int_0^\tau\int_{\F(t)} (r-\vr)\vU\cdot\nabla_x P'(r) \dxdt.
\end{align}
Adding $\int_0^{\tau}\mathcal{I}_4(t)\dt$ to the resulting expression of \eqref{eq:bubu100} we get
\begin{equation}\label{eq:bubu102}
    \int_0^\tau\int_{\F(t)} (r-\vr)\left(\de_t P'(r) + \vU\cdot\nabla_x P'(r)\right) \dxdt,
\end{equation}
and we use the continuity equation of \eqref{eq:fluidtransformed} which is satisfied by the strong solution $(r,\vU)$ to get
\begin{align}\label{eq:bubu101}
    &\de_t P'(r) + \vU\cdot\nabla_x P'(r) = P''(r)\left(\de_t r + \vU\cdot\nabla_x r\right) \\ \nonumber 
    & \quad = P''(r)( G(r,\vU) - r \Div \vU) = - p'(r) \Div \vU + P''(r) G(r,\vU). 
\end{align}
Plugging in \eqref{eq:bubu101} into \eqref{eq:bubu102} we get
\begin{align}\label{eq:bubu103}
    &\mathcal{J}_2+\int_0^\tau [\mathcal{I}_4(t)+\mathcal{I}_5(t)]\dt \nonumber \\ &=\int_0^\tau\int_{\F(t)} -p'(r)(r-\vr)\Div \vU + P''(r)(r-\vr)G(r,\vU) \dxdt,
\end{align}
and we combine the first integral with $\mathcal{I}_3$ to get
\begin{equation*}
    \left|\int_0^\tau\int_{\F(t)} \Div \vU\left(p(\vr) - p(r) - p'(r)(\vr-r)\right) \dxdt\right| \leq \int_0^\tau h_3(t)\mathcal{E}([\vr,\vu]|[r,\vU]) \dt.
\end{equation*}
Combining the above estimates we arrive at 
\begin{align} \label{eq:e1}
\int_0^{\tau} \Big[\mathcal{I}_1^F(t)+\sum_{i=2}^5 \mathcal{I}_i(t)\Big]\dt \leq & \delta \left\|\tn{S}(\nabla_x\bU-\nabla_x\vu)\right\|^2_{L^2(Q_\F)} + C(\delta)\int_0^\tau h_4(t)\mathcal{E}([\vr,\vu]|[r,\vU]) \dt \nonumber\\
& +\int_0^\tau \int_{\F(t)} P''(r)(r-\vr)G(r,\vU)\dx\dt.
\end{align}

Now we would like to combine $\mathcal{I}_1^B$ with the boundary integrals $\mathcal{I}_6$ and $\mathcal{J}_5$. To this end we use the equations satisfied by the transformed rigid velocity \eqref{eq:rigidtransformed1}-\eqref{eq:rigidtransformed2}. We first use \eqref{eq:trick} again and estimate the term quadratic in $(\vu-\vU)$ from $\mathcal{I}_1^B$ in a similar manner as \eqref{eq:trick2}. In order to express the term
\begin{equation}
    \int_0^\tau \int_{\B(t)} \vr(\de_t\vU + \vU\cdot\nabla_x\vU)\cdot(\vU-\vu) \dxdt
\end{equation}
we follow the procedure described in the Appendix, especially see \eqref{eq:Ap1}-\eqref{eq:Ap7}. Note however that we have to use equations \eqref{eq:rigidtransformed1}-\eqref{eq:rigidtransformed2} containing additional terms. Therefore we end up with
\begin{align} \label{eq:e2}
    &\int_0^\tau \mathcal{I}_1^B(t)\dt = \int_0^\tau \int_{\B(t)} \vr(\de_t\vU + \vU\cdot\nabla_x\vU)\cdot(\vU-\vu) \dxdt \\ \nonumber
    &\quad = -\int_0^\tau \int_{\de\B(t)} (\tn{S}(\nabla_x\vU)-p(r)\tn{I})\vn\cdot(\vU-\vu)\,\d S \dt \\ \nonumber
    &\quad -\int_0^\tau m ((\vw^s-\vw_1)\times\vV^s)\cdot (\vV^s-\vV_1) - \tn{J}_1((\vw^s-\vw_1)\times\vw^s)\cdot(\vw^s-\vw_1) \dt.
\end{align}
Clearly we have 
\begin{align*}
    &\left|\int_0^\tau m ((\vw^s-\vw_1)\times\vV^s)\cdot (\vV^s-\vV_1) - \tn{J}_1((\vw^s-\vw_1)\times\vw^s)\cdot(\vw^s-\vw_1) \dt\right| \\ \nonumber
    &\quad \leq \int_0^\tau h_4(t)\mathcal{E}([\vr,\vu]|[r,\vU]) \dt,
\end{align*}
therefore \eqref{eq:e2} implies  
\begin{equation} \label{eq:e2b}
\left| \int_0^\tau [\mathcal{I}_1^B(t)+\mathcal{I}_6(t)]\dt+\mathcal{J}_5  \right| \leq \int_0^\tau h_4(t)\mathcal{E}([\vr,\vu]|[r,\vU]) \dt.   
\end{equation}

Finally, it remains to estimate the integrals over the fluid part containing the terms $\vc{F}(r,\vU)$ and $G(r,\vU)$ coming from \eqref{eq:bubu0} and \eqref{eq:e1}, more precisely
\begin{equation}\label{eq:starterrorterms}
    \int_0^\tau \int_{\F(t)} \frac{\vr}{r} \vc{F}(r,\vU)\cdot(\vU-\vu) - \frac{\vr}{r}G(r,\vU)\vU\cdot(\vU-\vu) + P''(r)(r-\vr)G(r,\vU) \dxdt.
\end{equation}
We split all integrals into three domains and we study first the set $\vr$ close to $r$, then $\vr$ small and finally $\vr$ large. To this end we set $r_- = \inf_{Q_\F} r(t,\vc{x}) > 0$ and $r_+ = \sup_{Q_\F} r(t,\vc{x}) < \infty$ and we recall that
\begin{align}\label{eq:pr26}
P(\vr) - P'(r)(\vr-r) - P(r) &\geq c(r)(\vr-r)^2 \quad \text{ for } \frac{r_-}{2} < \vr < 2r_+, \\ \nonumber
&\geq c(r)(1+\vr^\gamma) \quad \text{ otherwise}.
\end{align}
In order to simplify notation we set 
\begin{align*}
    S_1(t) &= \left\{\vc{x} \in Q_\F: \quad \frac{r_-}{2} < \vr(t,\vc{x}) < 2r_+ \right\},\\
    S_2(t) &= \left\{\vc{x} \in Q_\F: \quad \vr(t,\vc{x}) \leq \frac{r_-}{2} \right\},\\
    S_3(t) &= \left\{\vc{x} \in Q_\F: \quad \vr(t,\vc{x}) \geq 2r_+ \right\}.
\end{align*}

We start with the last term of \eqref{eq:starterrorterms}. Using the expression \eqref{eq:Gdef} for $G(r,\vU)$, Lemma \ref{l:31} and \eqref{eq:pr26} we have
\begin{align} \label{eq:estimaterhomed}
    &\left|\int_{S_1(t)} P''(r)(r-\vr)G(r,\vU) \dx \right| \leq C\|r-\vr\|_{L^2(S_1(t))} \|\de_t \ZZ\|_{L^\infty(\F(t))} \\ \nonumber
    &\quad \leq C(\|r-\vr\|_{L^2(S_1(t))}^2 + |\vV^s-\vV_1|^2 + |\vw^s-\vw_1|^2) \leq C\mathcal{E}([\vr,\vu]|[r,\vU])(t).
\end{align}
In the case of small $\vr$ we proceed as follows.
\begin{align} \label{eq:estimaterhosmall}
    &\left|\int_{S_2(t)} P''(r)(r-\vr)G(r,\vU) \dx \right| \leq C\|1\|_{L^2(S_2(t))} \|\de_t \ZZ\|_{L^\infty(\F(t))} \\ \nonumber
    &\quad \leq C(\|1\|_{L^2(S_2(t))}^2 + |\vV^s-\vV_1|^2 + |\vw^s-\vw_1|^2) \leq C\mathcal{E}([\vr,\vu]|[r,\vU])(t).
\end{align}
Finally for large $\vr$ we split the estimate into two cases. First for $\gamma \leq 2$ we have
\begin{align} \label{eq:estimaterholarge1}
    &\left|\int_{S_3(t)} P''(r)(r-\vr)G(r,\vU) \dx \right| \leq C\|\vr\|_{L^\gamma(S_3(t))} \|\de_t \ZZ\|_{L^\infty(\F(t))} \\ \nonumber
    &\quad \leq C(\|\vr\|_{L^\gamma(S_3(t))}^2 + |\vV^s-\vV_1|^2 + |\vw^s-\vw_1|^2) \leq C\left(\mathcal{E}([\vr,\vu]|[r,\vU])^{2/\gamma}(t) + \mathcal{E}([\vr,\vu]|[r,\vU])(t) \right) \\ \nonumber
    &\quad \leq C\mathcal{E}([\vr,\vu]|[r,\vU])(t),
\end{align}
where we have used $\mathcal{E}([\vr,\vu]|[r,\vU]) \in L^\infty(0,T)$ and the fact that $\frac{2}{\gamma} \geq 1$. In the case $\gamma > 2$ we use $\vr \leq C\vr^{\gamma/2}$ to get
\begin{align} \label{eq:estimaterholarge2}
    &\left|\int_{S_3(t)} P''(r)(r-\vr)G(r,\vU) \dx \right| \leq C\|\vr^{\gamma/2}\|_{L^2(S_3(t))} \|\de_t \ZZ\|_{L^\infty(\F(t))} \\ \nonumber
    &\quad \leq C(\|\vr\|_{L^\gamma(S_3(t))}^\gamma + |\vV^s-\vV_1|^2 + |\vw^s-\vw_1|^2) \leq  C\mathcal{E}([\vr,\vu]|[r,\vU])(t).
\end{align}
Combining \eqref{eq:estimaterhomed}-\eqref{eq:estimaterholarge2} we get 
\begin{equation} \label{eq:e3}
\left| \int_{\F(t)} P''(r)(r-\vr)G(r,\vU) \dx \right| \leq C\mathcal{E}([\vr,\vu]|[r,\vU])(t).    
\end{equation}
Next we estimate the second term of \eqref{eq:starterrorterms}. 
For this purpose we split it into two parts as in \eqref{eq:bubu1} using
\begin{equation*}
    \frac{\vr}{r} = 1 + \frac{\vr-r}{r}.
\end{equation*}
Then we have
\begin{align*}
    &\left|\int_{S_1(t)} G(r,\vU)\vU\cdot(\vU-\vu) \dx \right| \leq C\|\vU-\vu\|_{L^6(S_1(t))} \|\de_t \ZZ\|_{L^\infty(\F(t))} \\ \nonumber
    &\quad \leq \delta\|\tn{S}(\nabla_x\vU - \nabla_x\vu)\|_{L^2(S_1(t))}^2 + C(\delta)(|\vV^s-\vV_1|^2 + |\vw^s-\vw_1|^2) \\ \nonumber
    &\quad \leq \delta\|\tn{S}(\nabla_x\vU - \nabla_x\vu)\|_{L^2(\F(t))}^2 + C\mathcal{E}([\vr,\vu]|[r,\vU])(t).
\end{align*}
The same term with density is treated as
\begin{align*}
    &\left|\int_{S_1(t)} \frac{\vr-r}{r}G(r,\vU)\vU\cdot(\vU-\vu) \dx \right| \leq C\|\vU-\vu\|_{L^6(S_1(t))}\|\vr-r\|_{L^2(S_1(t))}\|\de_t \ZZ\|_{L^\infty(\F(t))} \\ \nonumber
    &\quad \leq \delta\|\tn{S}(\nabla_x\vU - \nabla_x\vu)\|_{L^2(S_1(t))}^2 + C(\delta)\|\vr-r\|_{L^2(S_1(t))}^2(|\vV^s-\vV_1|^2 + |\vw^s-\vw_1|^2) \\ \nonumber
    &\quad \leq \delta\|\tn{S}(\nabla_x\vU - \nabla_x\vu)\|_{L^2(\F(t))}^2 + C\mathcal{E}([\vr,\vu]|[r,\vU])(t),
\end{align*}
where we again used that $\mathcal{E}([\vr,\vu]|[r,\vU]) \in L^\infty(0,T)$ and so  $\mathcal{E}([\vr,\vu]|[r,\vU])^2(t) \leq C \mathcal{E}([\vr,\vu]|[r,\vU])(t)$.

The case $\vr$ small is estimated similarly as in \eqref{eq:estimaterhosmall} and the case $\vr$ large is again estimated separately for $\gamma \leq 2$ and for $\gamma > 2$ as in \eqref{eq:estimaterholarge1} and \eqref{eq:estimaterholarge2}, respectively. Note that since $\gamma > \frac 32$, we do not run into trouble with the H\"older inequality. This way we get 
\begin{equation} \label{eq:e4}
\int_{\F(t)}\frac{\vr}{r}G(r,\vU)\vU\cdot(\vU-\vu)\dx \leq \delta\|\tn{S}(\nabla_x\vU - \nabla_x\vu)\|_{L^2(\F(t))}^2 + C\mathcal{E}([\vr,\vu]|[r,\vU])(t).     
\end{equation}

Finally, we estimate the first term of \eqref{eq:starterrorterms}. Here we proceed the same way as in the second term provided we show the estimate of $\|\vF(r,\vU)\|_{L^\infty(\F(t))}$. To this end we use formula \eqref{eq:Fdef} and use here the notation 
\begin{equation*}
    \vF(r,\vU) = \sum_{k = 1}^{11} \vF^k(r,\vU),
\end{equation*}
where $\vF^k(r,\vU)$ are the 11 terms of the right hand side of \eqref{eq:Fdef}.
In order to estimate these terms we apply Lemma \ref{l:31}.
In the first three terms we use \eqref{eq:318} to end up with
\begin{equation*}
    \sum_{k=1}^3 \|\vF^k\|_{L^\infty(\F(t))} \leq C\|\de_t \ZZ\|_{W^{1,\infty}(\F(t))} \leq C(|\vV^s-\vV_1| + |\vw^s-\vw_1|)(t).
\end{equation*}
Next, using \eqref{eq:Gammadef} and \eqref{eq:317} it is easy to conclude that
\begin{align*}
    & \|\vF^4 + \vF^9 + \vF^{10}\|_{L^\infty(\F(t))} \leq C\|\Gamma^a_{bc} \|_{L^{\infty}(\F(t))} \leq
    C\|\nabla^2_x \ZZ \|_{L^{\infty}(\F(t))}  \\ \nonumber
    &\qquad \leq C(\|\vV^s-\vV_1\|_{L^2(0,t)} + \|\vw^s-\vw_1\|_{L^2(0,t)}).
\end{align*}
Similarly we have
\begin{equation*}
    \|\vF^{11}\|_{L^\infty(\F(t))} \leq     C\|\nabla^3_x \ZZ \|_{L^{\infty}(\F(t))}  \leq C(\|\vV^s-\vV_1\|_{L^2(0,t)} + \|\vw^s-\vw_1\|_{L^2(0,t)}).
\end{equation*}

In order to estimate other terms we claim that the analog of \eqref{eq:317} holds also for $\ZZj$, namely that
\begin{equation} \label{eq:317new}
     \|\ZZj(t,\widetilde{\bZ}_2(t,\cdot)) - \mathrm{id}\|_{W^{2,\infty}(\F(t))} \leq C\left(\|\vV_1-\vV^s\|_{L^2(0,t)} + \|\vw_1-\vw^s\|_{L^2(0,t)}\right).
\end{equation}
Indeed, this is not difficult to observe by similar arguments to those used in the proof of Lemma \ref{l:31}. Using \eqref{eq:317new} we easily get
\begin{equation*}
    \|\Delta_x \ZZj \|_{L^\infty(\F(t))} + \|\tn{G}-\tn{I}\|_{L^\infty(\F(t))} \leq C\left(\|\vV_1-\vV^s\|_{L^2(0,t)} + \|\vw_1-\vw^s\|_{L^2(0,t)}\right), 
\end{equation*}
and consequently
\begin{equation*}
    \|\vF^5 + \vF^6 + \vF^7 + \vF^8 \|_{L^\infty(\F(t))} \leq C\left(\|\vV_1-\vV^s\|_{L^2(0,t)} + \|\vw_1-\vw^s\|_{L^2(0,t)}\right) .
\end{equation*}
Combining the estimates for $\vF^k$ with the arguments used in the proof of \eqref{eq:e4} we obtain 
\begin{equation} \label{eq:e5}
\left|\int_{\F(t)}\frac{\vr}{r}F(r,\vU)\cdot(\vU-\vu)\dx\right| \leq \delta\|\tn{S}(\nabla_x\vU - \nabla_x\vu)\|_{L^2(\F(t))}^2 + C\mathcal{E}([\vr,\vu]|[r,\vU])(t).     
\end{equation}
Putting together the estimates \eqref{eq:e1}, \eqref{eq:e2b}, \eqref{eq:e3}, \eqref{eq:e4} and \eqref{eq:e5} we finally end up with 
\begin{equation*} 
\left| \int_0^\tau \mathcal{R}(\vr,\vu,r,\vU)(t) \right| \leq \delta \left\|\tn{S}(\nabla_x\bU-\nabla_x\vu)\right\|^2_{L^2(Q_\F)} + C(\delta)\int_0^\tau h_5(t)\mathcal{E}([\vr,\vu]|[r,\vU])(t) \dt.    
\end{equation*}
Combining this estimate with \eqref{eq:e0} we obtain for sufficiently small $\delta>0$
\begin{equation*}
    \mathcal{E}([\vr,\vu]|[r,\vU])(\tau) \leq \int_0^\tau h(t)\mathcal{E}([\vr,\vu]|[r,\vU])(t) \dt
\end{equation*}
with some nonnegative function $h \in L^1(0,T)$ and use the Gronwall lemma to conclude that 
\begin{equation}
\mathcal{E}([\vr,\vu]|[r,\vU])(\tau) = 0    
\end{equation}
for almost all $\tau \in (0,T)$. This in particular implies $\vV_1 = \vV^s$, $\vw_1 = \vw^s$, $\vr = r$ and $\vu = \vU$.

In order to finish the proof we need to show that in fact $\vV_1 = \vV_2$, $\vw_1 = \vw_2$, $\B_1(t) = \B_2(t)$ and that $(\vr_{1\F},\vu_{1\F}) = (\vr_{2\F},\vu_{2\F})$. Using \eqref{eq:strongbody} we have $\vV_1 =\tn{O}_1\tn{O}_2^T\vV_2$ and $\vw_1 = \tn{O}_1\tn{O}_2^T\vw_2$, i.e.
\begin{align} \label{eq:121}
    \tn{O}_1^T\vV_1 &= \tn{O}_2^T\vV_2, \\
    \tn{O}_1^T\vw_1 &= \tn{O}_2^T\vw_2. \label{eq:122}
\end{align}
Moreover, we have using \eqref{eq:QO} that
\begin{equation}
    \vw_i \times \x = \tn{Q}_i\x = \frac{\d \tn{O}_i}{\dt}\tn{O}_i^T\x
\end{equation}
for $i=1,2$ and for all $\x \in \R^3$. Therefore
\begin{equation}
    \tn{O}_i^T\vw_i \times \x = \tn{O}_i^T(\vw_i \times \tn{O}_i\x) = \tn{O}_i^T\frac{\d \tn{O}_i}{\dt}\tn{O}_i^T\tn{O}_i\x = \tn{O}_i^T\frac{\d \tn{O}_i}{\dt}\x
\end{equation}
for $i=1,2$ and all $\x \in \R^3$ and thus from \eqref{eq:122} we conclude that 
\begin{equation} \label{eq:458}
    \tn{O}_1^T\frac{\d \tn{O}_1}{\dt} = \tn{O}_2^T\frac{\d \tn{O}_2}{\dt}.
\end{equation}
Simple manipulation of \eqref{eq:458} yields
\begin{equation*}
    \frac{\d(\tn{O}_1 - \tn{O}_2)}{\dt} = (\tn{O}_1-\tn{O}_2)\tn{O}_2^T\frac{\d \tn{O}_2}{\dt}.
\end{equation*}
Denoting $\tn{O}_\Delta = \tn{O}_1 - \tn{O}_2$ and treating $\tn{O}_2$ as a given function of time we end up with the ordinary differential equation
\begin{align} \label{eq:ODE1}
    \frac{\d \tn{O}_\Delta(t)}{\dt} &= \tn{O}_\Delta(t) \tn{W}(t), \\
    \tn{O}_\Delta(0) &= 0, \label{eq:ODE2}
\end{align}
for some given matrix valued function $\tn{W}(t)$. The problem \eqref{eq:ODE1}-\eqref{eq:ODE2} has a unique solution $\tn{O}_\Delta(t) = 0$, i.e. $\tn{O}_1 = \tn{O}_2$ and from \eqref{eq:121} and \eqref{eq:122} we conclude $\vV_1 = \vV_2$ and $\vw_1 = \vw_2$. In particular we know that the position of the bodies of the weak and strong solutions are the same, i.e. $\B_1(t) = \B_2(t)$.

Finally it is easy to set the cutoff functions $\zeta_i(t,\x)$ ($i=1,2$) introduced in the beginning of Section \ref{s:cc} to coincide in the case that $\B_1 = \B_2$. Therefore $\Lambda_1(t,\x) = \Lambda_2(t,\x)$, $\bZ_1 = \bZ_2$ and $\ZZ(t,\x) = \ZZj(t,\x) = \x$ for all $\x \in \Omega$ and $t \in (0,T)$. Finally we use \eqref{newsol1} and \eqref{newsol2} to conclude that $\vr_{1\F} = \vr_{2\F}$ and $\vu_{1\F} = \vu_{2\F}$. 

Now we return to the notation $T_{min}$ introduced in \eqref{eq:tmin}. We have shown that the weak and strong solutions coincide on the time interval $(0,T_{min})$. However, if $T_{min} < T$, we have on one hand
\begin{equation*}
    d(\B_2(T_{min}),\de\Omega) \geq \kappa
\end{equation*}
due to \eqref{eq:dist_cond} and on the other hand
\begin{equation*}
    d(\B_1(T_{min}),\de\Omega) \leq \frac{\kappa}{2},
\end{equation*}
which is a contradiction. Therefore it has to hold $T_{min} = T$. Theorem \ref{t3} is proved.

\section{Appendix}

Here we present how to obtain the weak formulation of the momentum equation \eqref{moment} from the strong formulation \eqref{fluidmotion} and \eqref{bodymotion}.

First we recall the transport theorem used to express time derivatives of integrals over time dependent domains. We have
\begin{align} \label{trans}
    \frac{\d}{\d t} \int_{\F(t)} f(x,t) \dx &= \int_{\F(t)} \de_t f(x,t) + \Div (f \vu)(x,t) \dx \\
    & = \int_{\F(t)} \de_t f(x,t) \dx + \int_{\de \B(t)} f \vu\cdot \vc{n}(x,t)\, \d S. \nonumber
\end{align}
Here $\vu$ is the velocity describing the movement of the domain, which in our case is the solution itself, so $\vu = \vu_\F$ and moreover we know that $\vu_\F\cdot \vc{n} = \vu_\B\cdot \vc{n}$ on the boundary of the body $\B(t)$.

Let us take a test function  $\vphi \in {\mathcal R}(Q_{\B})$ and recall that
$${\mathcal R}(Q_{\B})\equiv \{\vphi\in C^\infty_c([0,T)\times \Omega)\ |\ \tn{D}(\vphi)=0\ \ \mbox{on an open neighbourhood of}\ \overline{Q_{\B}}\}.$$ 
Starting from the strong formulation of the momentum equation for the fluid we get 
\begin{equation}
\int_0^T \int_{\F(t)} \left(\de_t(\vr_\F\vu_\F) + \Div(\vr_\F\vu_\F\otimes\vu_\F) - \Div \tn{S}(\nabla_x \vu_\F) + \nabla_x p(\vr_\F) \right)\cdot\vphi \dxdt = 0.    
\end{equation}
By \eqref{trans}, the time derivative term yields
\begin{align}\label{eq:TD1}
    &\int_0^T \int_{\F(t)}\de_t(\vr_\F\vu_\F)\cdot \vphi \dxdt = \int_0^T \int_{\F(t)}\de_t(\vr_\F\vu_\F\cdot \vphi) - \vr_\F\vu_\F\cdot\de_t\vphi \dxdt \\ \nonumber
    &\quad =  - \int_{\F(0)}(\vr_\F\vu_\F\cdot\vphi)(0,x)\dx \\ \nonumber
    &\quad - \int_0^T \int_{\F(t)} \vr_\F\vu_\F\cdot\de_t\vphi\dxdt - \int_0^T \int_{\F(t)} \Div((\vr_\F\vu_\F\cdot\vphi)\vu_\F)\dxdt.
\end{align}
Moreover the convective term yields
\begin{align}\label{eq:CT1}
    &\int_0^T \int_{\F(t)}\Div(\vr_\F\vu_\F\otimes\vu_\F)\cdot \vphi \dxdt \\ \nonumber
    &\quad = \int_0^T \int_{\F(t)} \Div((\vr_\F\vu_\F\cdot\vphi)\vu_\F)\dxdt - \int_0^T \int_{\F(t)} (\vr_\F\vu_\F\otimes\vu_\F):\nabla_x\vphi\dxdt.
\end{align}
Finally, the stress tensor and the pressure terms yield simply
\begin{align}
    &- \int_0^T \int_{\F(t)}\Div(\tn{S}(\nabla_x\vu_\F) - p(\vr_\F)\tn{I})\cdot \vphi \dxdt \\ \nonumber
    &\quad = \int_0^T \int_{\F(t)} (\tn{S}(\nabla_x\vu_\F) - p(\vr_\F)\tn{I}):\nabla_x\vphi\dxdt - \int_0^T \int_{\de \B(t)} (\tn{S}(\nabla_x\vu_\F) - p(\vr_\F)\tn{I})\vc{n}\cdot\vphi\,\d S \dt.
\end{align}
Plugging all the above computations together we end up with the weak formulation of the momentum equation on the fluid domain as 
\begin{align}\label{eq:MEfluidpart}
    & - \int_{\F(0)}(\vr_\F\vu_\F\cdot\vphi)(0,x)\dx - \int_0^T \int_{\de \B(t)} (\tn{S}(\nabla_x\vu_\F) - p(\vr_\F)\tn{I})\vc{n}\cdot\vphi\,\d S \dt\\ \nonumber
    &\quad = \int_0^T \int_{\F(t)} \vr_\F\vu_\F\cdot\de_t\vphi + (\vr_\F\vu_\F\otimes\vu_\F - \tn{S}(\nabla_x\vu_\F)  + p(\vr_\F)\tn{I}):\nabla_x\vphi \dxdt.
\end{align}

Now let us turn our attention to the rigid body. Since $\tn{D}(\vphi) = 0$ on $\B(t)$, there exist functions $\vc{A}(t)$ and $\vc{B}(t)$ such that
\begin{align*}
\vu_\B(t,x) &= \vV(t) + \vw(t) \times (\x - \X(t)) \qquad \text{ in } Q_\B \\
\vphi(t,x) &= \vc{A}(t) + \vc{B}(t) \times (\x - \X(t)) \qquad \text{ in } Q_\B.    
\end{align*}
It is a matter of a simple calculation to deduce that
\begin{equation} \label{eq:Ap0}
\frac{\de \vu_\B}{\de t} + \vu_\B\cdot\nabla_x\vu_\B = \frac{\d \vV}{\d t} + \frac{\d \vw}{\d t}\times (\x-\X) + \vw \times (\vw \times (\x-\X)) \quad \text{ in } Q_\B.
\end{equation}
We use \eqref{bodymotion}$_1$ to deduce 
\begin{equation}\label{eq:Ap1}
    \int_{\B(t)} \vr_\B \frac{\d \vV}{\d t}\cdot \vc{A} \dx = - \int_{\de \B(t)} (\tn{S}(\nabla_x\vu_\F) - p(\vr_\F)\tn{I})\vc{n}\cdot \vc{A}\,\d S. 
\end{equation}    
Next, by \eqref{def:X} we get
\begin{align}    
    &\int_{\B(t)} \vr_\B \frac{\d \vV}{\d t}\cdot (\vc{B}\times (\x-\X)) \dx = \left(\frac{\d \vV}{\d t} \times \vc{B}\right)\cdot \int_{\B(t)} \vr_\B(\x-\X) \dx = 0,\\
    &\int_{\B(t)} \vr_\B \left(\frac{\d \vw}{\d t} \times (\x-\X)\right)\cdot \vc{A} \dx = \vc{A}\cdot \left(\frac{\d \vw}{\d t} \times \int_{\B(t)} \vr_\B(\x-\X) \dx\right) = 0, \\
    &\int_{\B(t)} \vr_\B (\vw\times(\vw\times(\x-\X))\cdot \vc{A} \dx = \vc{A}\cdot \left(\vw\times \left(\vw\times \int_{\B(t)} \vr_\B(\x-\X) \dx\right)\right) = 0.
\end{align}
Finally, using \eqref{def:J} and \eqref{bodymotion}$_2$ we obtain
\begin{align}
    &\int_{\B(t)} \vr_\B \left(\frac{\d \vw}{\d t} \times (\x-\X)\right)\cdot(\vc{B}\times (\x-\X)) \dx  = \tn{J}\frac{\d \vw}{\d t}\cdot \vc{B} \\ \nonumber
    & \quad = (\tn{J}\vw \times \vw)\cdot \vc{B} - \vc{B}\cdot \int_{\de \B(t)} (\x-\X)\times (\tn{S}(\nabla_x\vu_\F) - p(\vr_\F)\tn{I})\vc{n}\,\d S \\ \nonumber
    & \quad = \tn{J}\vw \cdot (\vw\times\vc{B}) - \int_{\de \B(t)} (\tn{S}(\nabla_x\vu_\F) - p(\vr_\F)\tn{I})\vc{n} \cdot (\vc{B}\times (\x-\X))\,\d S,
\end{align}
and by a more lengthy computation one can also conclude
\begin{equation}\label{eq:Ap6}
    \int_{\B(t)} \vr_\B\left(\vw\times \left(\vw \times (\x-\X)\right)\right)\cdot(\vc{B}\times (\x-\X)) \dx = - \tn{J}\vw\cdot (\vw\times\vc{B}).
\end{equation}
Summing \eqref{eq:Ap1}-\eqref{eq:Ap6} and using \eqref{eq:Ap0} we end up with 
\begin{equation}\label{eq:Ap7}
    \int_{\B(t)} \vr_\B(\de_t \vu_\B + \vu_\B\cdot\nabla_x\vu_\B) \cdot \vphi \dx = - \int_{\de \B(t)}  (\tn{S}(\nabla_x\vu_\F) - p(\vr_\F)\tn{I})\vc{n}\cdot\vphi \,\d S,
\end{equation}
which together with the conservation of mass of the body \eqref{eq:CE_body} yields
\begin{equation}\label{eq:bla0}
    \int_{\B(t)} (\de_t(\vr_\B \vu_\B) + \Div(\vr_\B\vu_\B\otimes\vu_\B)) \cdot \vphi \dx = - \int_{\de \B(t)}  (\tn{S}(\nabla_x\vu_\F) - p(\vr_\F)\tn{I})\vc{n}\cdot\vphi \,\d S.
\end{equation}
Now, similarly as in \eqref{eq:TD1}, \eqref{eq:CT1} we have
\begin{align}\label{eq:bla1}
    &\int_0^T \int_{\B(t)} (\de_t(\vr_\B \vu_\B) + \Div(\vr_\B\vu_\B\otimes\vu_\B)) \cdot \vphi \dxdt \\ \nonumber
    & \quad =  - \int_{\B(0)}(\vr_\B\vu_\B\cdot\vphi)(0,x)\dx - \int_0^T \int_{\B(t)} \vr_\B\vu_\B\cdot\de_t\vphi \dxdt + \int_0^T \int_{\B(t)} (\vr_\B\vu_\B\otimes\vu_\B):\nabla_x\vphi\dxdt,
\end{align}
where the last term on the right hand side is zero, because $\tn{D}(\vphi) = 0$ on $\B(t)$. This is also the reason why
\begin{equation}\label{eq:bla2}
    \int_0^T \int_{\B(t)} \tn{S}(\nabla_x\vu_\B):\nabla_x\vphi - p(\vr_B) \Div \vphi \dxdt = 0.
\end{equation}
Combining \eqref{eq:bla0} - \eqref{eq:bla2} we end up with 
\begin{align}\label{eq:MEsolidpart}
    & - \int_{\B(0)}(\vr_\B\vu_\B\cdot\vphi)(0,x)\dx + \int_0^T \int_{\de \B(t)} (\tn{S}(\nabla_x\vu_\F) - p(\vr_\F)\tn{I})\vc{n}\cdot\vphi\,\d S \dt\\ \nonumber
    &\quad = \int_0^T \int_{\B(t)} \vr_\B\vu_\B\cdot\de_t\vphi + (\vr_\B\vu_\B\otimes\vu_\B - \tn{S}(\nabla_x\vu_\B)  + p(\vr_\B)\tn{I}):\nabla_x\vphi \dxdt,
\end{align}
which together with \eqref{eq:MEfluidpart} yields \eqref{moment}.

\section*{Acknowledgements}
 The work of O. Kreml and \v S. Ne\v casov\'a was supported by the Czech Science Foundation grant GA19-04243S in the framework of RVO 67985840. The work of T. Piasecki was supported by Polish National Science Centre grant 2018/29/B/ST1/00339 and by the Czech Science Foundation grant GA19-04243S in the framework of RVO 67985840.


\end{document}